\documentclass[12pt]{amsart}
\usepackage{amstext,amsfonts,amssymb,amsbsy,amsthm,array,graphicx}
\usepackage{stmaryrd,txfonts}
\newcommand{\bfx}{{\mathbf x}}

\newcommand{\tnorm}[1]{{{|||} #1 {|||}}}
\newcommand{\jump}[1]{\llbracket #1 \rrbracket }
\def\eref#1{{\rm (\ref{#1})}}
\newtheorem{theorem}{Theorem}
\newtheorem{lemma}{Lemma}
\newtheorem{proposition}{Proposition}
\newtheorem{remark}{Remark}
\newtheorem{corollary}{Corollary}

\begin{document}

\title[Solving ill-posed control
problems by stabilized
  FEM]{Solving ill-posed control
problems by stabilized
  finite element methods: an alternative to Tikhonov regularization}

\author{Erik Burman$^1$, Peter Hansbo$^2$, Mats G. Larson$^3$}
 \address{$^1$Department of Mathematics, University College London, London, UK--WC1E 6BT, United Kingdom}
 \address{$^2$Department of Mechanical Engineering, J\"onk\"oping University,
SE-55111 J\"onk\"oping, Sweden}
 \address{$^3$Department of Mathematics and Mathematical Statistics,
Ume{\aa} University, SE--901 87 Ume{\aa}, Sweden}

\begin{abstract}
Tikhonov regularization is one of the most commonly used methods of regularization of ill-posed problems. In the setting
of finite element solutions of elliptic partial differential control problems, Tikhonov regularization amounts to adding suitably weighted
least squares terms of the control variable, or derivatives thereof, to the Lagrangian determining the optimality system.
In this note we show that stabilization methods for discretely ill--posed problems developed in the setting
of convection--dominated convection--diffusion problems, can be highly
suitable for stabilizing optimal control problems, and that Tikhonov
regularization will lead to less accurate discrete solutions. We
consider data assimilation problems for Poisson's equation as
illustration and derive new error estimates both for the the
reconstruction of the solution from measured data and reconstruction
of the source term from measured data. These estimates include both
the effect of discretization error and error in measurements.
\newline\newline
Keywords: optimal control problem,  source identification, finite elements, regularization
\end{abstract}

\maketitle

\section{Introduction}
In this note we propose an alternative to the classical Tikhonov regularization approach
in finite element approximations of optimal control problems governed by elliptic partial differential equations. We shall, following \cite{BKR00}, consider problems of the type
\begin{equation}
J(u,q)\rightarrow\text{min!},\quad A(u)=f+B(q),
\end{equation}
where  $J$ is a cost functional, $A$ is an elliptic differential operator for the state variable $u$, 
and $B$ an impact operator for the control variable $q$. Introducing costate variables $\lambda$, this problem can be formulated as finding saddle points to the Lagrangian functional
\begin{equation}
{\mathfrak L}(u,q,\lambda) := J(u,q) + (\lambda, A(u)-f-B(q)), 
\end{equation}
where $(\cdot,\cdot)$ denotes the $L_2$ inner product,
determined by the system
\begin{equation}
\left\{\begin{array}{>{\displaystyle}l}
\frac{d}{d\epsilon_1}\left.{\mathfrak L}(u+\epsilon_1 v,q,\lambda)\right|_{\epsilon_1 = 0} = 0 ,\\[2mm]
\frac{d}{d\epsilon_2}\left.{\mathfrak L}(u,q+\epsilon_2 r,\lambda)\right|_{\epsilon_2 = 0} = 0 ,\\[2mm]
\frac{d}{d\epsilon_3}\left.{\mathfrak L}(u,q,\lambda +\epsilon_3 \mu)\right|_{\epsilon_3 = 0} = 0 .\\[2mm]
\end{array}\right.\end{equation}
In a finite element setting, the continuous states, controls, and costates $(u,q,\lambda)\in V\times Q\times V$ are replaced by their discrete counterparts $(u^h,q^h,\lambda^h)\in V^h\times Q^h\times V^h$,
where $V^h$ and $Q^h$ are finite dimensional counterparts of the appropriate Hilbert spaces $V$ and $Q$, respectively.

Typically, the cost functional measures some distance between the discrete state and a known or sampled
function $u_0$ over a subdomain $\mathcal{M} \subseteq \Omega$ where
$\Omega \subset \mathbb{R}^d$, $d=2,3$ is a polyhedral (polygonal) domain of computation.
\begin{equation}
J(u,q) := \frac12 \| u- u_0\|^2_{\mathcal{M}}
\end{equation}
which may not lead to a well posed problem. A classical regularization method due to Tikhonov, see \cite{ItJi15}, is to add
a stabilizing functional $n(q,q)$,
\begin{equation}
J(u,q) := \frac12 \| u- u_0\|^2_{\mathcal{M}} + n(q,q) ,
\end{equation}
where, typically,
\begin{equation}
n(q,q) :=  \alpha \| q - q_b\|^2 +\beta  \| \nabla (q - q_b)\|^2 ,
\end{equation}
with $\alpha$ and $\beta$ regularization parameters and $q_b$ the background state, or first guess state. The role of the background state
is to diminuish the nonconsistent character of the Tikhonov
regularization and implies additional a priori knowledge on the system
beyond the samples $u_0$. In this note we will assume that no such
additional a priori data are at hand, $q_b=0$ and that there is no
physical justification for the addition of the term $n(q,q)$, or that the size
of $\alpha$ and $\beta$ given by the application are too small to
provide sufficient stabilization of the system for computational
purposes. The objective of
this note is to show that these regularizations can be improved upon
in a finite element framework. The approach that we will follow is to eliminate the Tikhonov regularization on
the continuous level and instead regularize the discrete formulation 
hence making the regularization part of the computational
method in the form of a weakly consistent stabilization.

The terminology stabilization versus regularization is slightly ambiguous but in classical numerical analysis
the method of modified equations \cite{GrSa86} provides a link between these concepts.

\subsection{Model problems}
We will discuss two model problems below. First let $\mathcal{M}
\subset \Omega$ be a ball $B_{r_1}(\bf{x}_0)$ with radius $r_1$,
centered at $\bf{x}_0$ and assume that
measurements $u_0$ are available in this ball. We then wish to
reconstruct the solution $u$ in $B_{r_2}({\bf{x}}_0)$ under the a
priori assumption that $B_{r_2}({\bf{x}}_0) \subset
B_{r_3}({\bf{x}}_0)\subset \Omega$ and $u \in H^1(\Omega)$, is a weak solution to 
\[
-\Delta u = f, \mbox{ in } \Omega.
\]
This problem can be cast in the form of a
constrained minimization problem:
\begin{equation}\label{eq:minim}
\frac12\| u-u_0 \|^2_{L_2(\mathcal{M})} + \frac{\alpha}{2} \| \nabla u \|^2_{L_2(\Omega)}
\end{equation}
subject to
\begin{equation}\label{Poisson}
-\Delta u = f, \mbox{ in } \Omega.
\end{equation}
We will refer to this problem as the \emph{data assimilation problem}
below, but it is also strongly related to boundary control problems.
It is known that in this case if for $\alpha=0$ $u_0$ is such that a
solution exists, then by unique continuation of harmonic functions
this solution is unique. This statement can be quantified in the
following three sphere's inequality:
\begin{lemma}\label{3sphere}(Three spheres inequality)
Assume that $u:\Omega \rightarrow \mathbb{R}$ is a weak solution of
\eref{Poisson} with $f \in H^{-1}(\Omega)$ such that
$\|f\|_{H^{-1}(\Omega)} \leq \varepsilon$ for some $\varepsilon>0$.
For every $r_1,\, r_2,\,r_3,\, \overline r$ such that
$0<r_1<r_2<r_3<\overline r$ and for every ${\bf{x}}_0 \in \Omega$ such that
$\mbox{dist}({\bf{x}}_0,\partial \Omega) > \overline r$ there holds
\begin{equation}\label{3sphere_stab}
\|u\|_{L^2(B_{r_2})} \leq C \left(\|u\|_{L^2(B_{r_1})} + \varepsilon\right)^\tau \cdot \left(\|u\|_{L^2(B_{r_3})}+ \varepsilon\right)^{(1-\tau)}
\end{equation}
where $B_{r_i}$, $i=1,2,3$ are balls centered at ${\bf{x}}_0$, $C>0$ and $\tau$, $0<\tau<1$ only depend on the geometry of $\Omega$, and the ratios
$r_2/r_1$ and $r_3/r_2$.
\end{lemma}
\begin{proof}
For a proof in the non-homogeneous case considered here see
Allessandrini et al. \cite[Theorem 1.10]{ARRV09}.
\end{proof}
\begin{remark}
We do not track the exact form of the geometric constants that appear in the proof of
Lemma \ref{3sphere}. They are all considered included in
the canonical constant $C$ above. For the precise definition of the
result with all exact dependencies we refer to \cite{ARRV09}.
\end{remark}
Assume that $\Omega \subset \mathbb{R}^d$ with $B_{r_3}(\bf{x}_0)
\subset \Omega$.
Introducing
the Lagrange multiplier $\lambda\in H^1_0(\Omega)$, we have the optimality system: find
$(u,\lambda)\in H^1(\Omega)\times H^1_0(\Omega)$ such that
\begin{equation}\label{eq:eq1a}
\int_{\mathcal{M}} u \, v \, d\Omega +\int_{\Omega} \nabla\lambda\cdot\nabla v\, d\Omega = \int_{\mathcal{M}} u_0\, v \, d\Omega\quad \forall v\in H^1(\Omega) ,
\end{equation}
\begin{equation}\label{eq:eq3a}
\int_{\Omega} \nabla u\cdot\nabla \mu \, d\Omega = \int_{\Omega} f\, \mu \, d\Omega\quad \forall \mu\in H^1_0(\Omega).
\end{equation}
The difficulty in this case is that the equation \eref{eq:eq3a} is
ill-posed, since $u \in H^1(\Omega)$ and $\mu \in H^1_0(\Omega)$. It
follows that coercivity fails for $u$ in $H^1(\Omega)$. Clearly this
is not the case when $\alpha>0$, but then a nonconsistent perturbation
is added to the system that can not easily be quantified. Below we
will instead add a regularization on the discrete level. Indeed we
will penalize the fluctuations of the gradient over element faces and
show that the added coercivity on the high frequency content of the
solution is sufficient to obtain a priori estimates leading to error
estimates through Lemma \ref{3sphere}. This part of the analysis
draws on earlier ideas for the elliptic Cauchy problem from
\cite{Bu13,Bu14b}. Below we will assume that $u_0$ is the unperturbed
measurement for which the unique solution exists and consider a
numerical method with perturbed data.

Our second example considers the case where the data is available in the
whole domain, $\mathcal{M}\equiv \Omega$, but the source term is
unknown and must be reconstructed. The challenge here being that the
application of the Laplacian is unstable. This case will be referred
to as \emph{source reconstruction} below, but is also related to
a distributed control problem.
We consider the elementary problem: minimize
\begin{equation}\label{eq:minim2}
\frac12\| u-u_0 \|^2_{L_2(\Omega)} + \frac{\alpha}{2} \| q \|^2_{L_2(\Omega)}
\end{equation}
subject to 
\begin{equation}\label{eq:constraint}
\text{$-\Delta u = q$ in $\Omega$, $u = 0$ on $\partial\Omega$}.
\end{equation}
Here, $u_0$ is known data and $\Omega$ is a convex polygonal
(polyhedral) subset of $\mathbb{R}^d$, $d=2,3$, with outword pointing
normal $n$. We assume that we wish to solve
\eref{eq:minim2}-\eref{eq:constraint} in the situation where $u_0$ is
a measurement on a system that is of the form \eref{eq:constraint}.
This means that if no perturbations are present in the data and measurements
are available in every point of $\Omega$, the minimizer for $\alpha=0$ is $u = u_0
\in H^1_0(\Omega) \cap
H^2(\Omega)=:W$
and an associated $q = -\nabla^2 u\in L^2(\Omega)$ exists so that
\eref{eq:constraint} is satisfied. Assume also that the shift theorem
$\|u\|_{H^2(\Omega)} \leq C \|q\|_{L^2(\Omega)}$ holds. Below we will
consider the problem of reconstructing $q$ from $u_0$ accounting for
both the discretization error and the error due to errors in the
measured data $u_0$.

Introducing
the Lagrange multiplier $\lambda\in H^1_0(\Omega)$, we have the optimality system: find
$(u,q,\lambda)\in H^1_0(\Omega)\times L_2(\Omega)\times H^1_0(\Omega)$ such that
\begin{equation}\label{eq:eq1}
\int_{\Omega} u \, v \, d\Omega +\int_{\Omega} \nabla\lambda\cdot\nabla v\, d\Omega = \int_{\Omega} u_0\, v \, d\Omega\quad \forall v\in H^1_0(\Omega) ,
\end{equation}
\begin{equation}\label{l2p}
\alpha \int_{\Omega} q \, r \, d\Omega +\int_{\Omega} \lambda\, r \, d\Omega = 0 \quad \forall r\in L_2(\Omega) ,
\end{equation}
\begin{equation}\label{eq:eq3}
\int_{\Omega} \nabla u\cdot\nabla \mu \, d\Omega = \int_{\Omega} q\, \mu \, d\Omega\quad \forall \mu\in H^1_0(\Omega).
\end{equation}
We note here that the trace of $\lambda$ is zero on the boundary but that this does not affect $q$ in the infinite dimensional case.
However, when we discretize and solve for discrete counterparts ($u^h,\lambda^h, q^h$), a problem arises in that the finite dimensionality of the
problem forces the zero boundary condition on $\lambda^h$ onto $q^h$
via the $L_2$--projection in (\ref{l2p}). This has profound
implications for the accuracy close to the boundary of the control $q^h$. As a remedy for
this we propose to introduce the regularization of $q$ in the discrete
setting so that it acts only on fluctuations of the gradient of the
source term, with a particular scaling in the mesh-size. Since the
kernel of this operator is so big, the stability of the system
\eref{eq:eq1}-\eref{eq:eq3} is compromised. We therefore also introduce a
stabilization of the type suggested above, where the jumps of
the gradient of $u$ are penalized as well. This is
sufficient to make the system inf-sup stable in a suitable
discrete norm as we shall see below.

\section{Derivation of the discrete model}

Let $\{\mathcal{T}_h\}_h$ denote a family of shape regular and quasi uniform
tesselations of $\Omega$ into nonoverlapping simplices, such that for
any two different simplices $K$, $K' \in \mathcal{T}_h$, $K \cap
K'$ consists of either the empty set, a common face/edge or a common
vertex. The outward pointing normal of a simplex $K$ will be
denoted $n_{K}$. We denote the set of interior element faces $F$ in $\mathcal{T}_h$ by
$\mathcal{F}_I$. To each face we associate a normal $n_F$ whose
orientation is arbitrary but fixed. We define the standard finite element space of continuous
piecewise affine functions on $\mathcal{T}_h$
\[
V_h := \{v_h \in C^0(\overline \Omega) : v_h \vert_K \in P_1(K), \, \forall K
\in \mathcal{T}_h\},
\]
where $P_1(K)$ denotes the set of affine functions on $K$. We also
define $V_h^0 := V_h \cap H^1_0(\Omega)$.
\subsection{Data assimilation}
Consider the discrete
formulation: find $u_h,\lambda_h \in V_h\times V_h^0$ such that
\begin{equation}\label{eq:da1}
m_{\mathcal{M}}(u_h,v_h)+s_1(u_h,v_h)+ 
a_h(v_h,\lambda_h)
= m_{\mathcal{M}}(\tilde u_0,v_h)\quad \forall v_h \in V_h,
\end{equation}
\begin{equation}\label{eq:da2}
a_h(u_h,\mu_h) = m_{\Omega}(
\tilde f, \mu_h )\quad \forall \mu_h \in V_h^0,
\end{equation}
where $\tilde f := f+\delta f$ and
$\tilde u_0 := u_0 + \delta u_0$,  with $\delta f \in H^{-1}(\Omega)$ and
$\delta u_0 \in L^2(\mathcal{M})$ denote
measurement errors in the source term and data. The bilinear forms are given by
\begin{equation}\label{mform}
m_{X}(u_h,v_h) := \int_{X} u_h \, v_h \, d\Omega, \mbox{ for } X
\subseteq \Omega,
\end{equation}
\begin{equation}\label{sform}
s_i(u_h,v_h):= \sum_{F \in \mathcal{F}_{I}}
\gamma \int_F h^i_F \jump{\nabla u_h \cdot n_F} \jump{\nabla v_h \cdot n_F}
~\mbox{d}s
\end{equation}
where $\jump{y_h}\vert_F := \lim_{\epsilon \rightarrow 0^+}
(y_h(x-\epsilon n_F) -y_h(x+\epsilon n_F))$ denotes the jump of the quantity $y_h$ over
the face $F$, with normal $n_F$ and finally
\begin{equation}\label{aform}
a_h(v_h,\lambda_h):=\int_{\Omega}
\nabla v_h\cdot\nabla \lambda_h\, d\Omega. 
\end{equation}
This may then be written on the compact form, find $u_h, \lambda_h \in \mathcal{V}^{DA}_h$, with $\mathcal{V}^{DA}_h :=  V_h \times V_h^0$ such that
\begin{equation*}
A_{DA}[(u_h,\lambda_h),(v_h,\mu_h)] = m_{\Omega}(
\tilde f, \mu_h ) + m_{\mathcal{M}}(\tilde u_0,v_h),\quad \forall
v_h,\mu_h \in   \mathcal{V}^{DA}_h,
\end{equation*}
with
\begin{equation*}\begin{array}{ll}
A_{DA}[(u_h,\lambda_h),(v_h,\mu_h)] := & m_{\mathcal{M}}(u_h,v_h)+s_1(u_h,v_h) \\[2mm]& +a_h(v_h,\lambda_h) + a_h(u_h,\mu_h).
\end{array}\end{equation*}
\subsection{Source reconstruction}
Consider the discrete
formulation: find $u_h,q_h,\lambda_h \in V_h^0 \times V_h\times V_h^0$ such that
\begin{equation}\label{eq:crazy1}
m(u_h,v_h)+s_1(u_h,v_h)+ 
a_h(v_h,\lambda_h)
= m(u_0,v_h)\quad \forall v_h \in V_h^0,
\end{equation}
\begin{equation}\label{eq:crazy2}
m(\lambda_h,w_h) - s_5(q_h,w_h) = 0 \quad \forall w_h \in V_h
\end{equation}
and
\begin{equation}\label{eq:crazy3}
a_h(u_h,\mu_h) = m(
q_h, \mu_h )\quad \forall \mu_h \in V_h^0.
\end{equation}
with bilinear forms given by \eref{mform}, \eref{sform} and
\eref{aform} above. Below we will distinguish the stabilization
parameters of $s_1(\cdot,\cdot)$ and $s_5(\cdot,\cdot)$ and then
denote them by $\gamma_1$ and $\gamma_5$ respectively.

This may then be written on the compact form, find $u_h, q_h, \lambda_h \in \mathcal{V}_h^{SR}$, with $\mathcal{V}_h^{SR} :=  V_h^0 \times V_h \times V_h^0$ such that
\begin{equation*}
A_{SR}[(u_h,q_h,\lambda_h),(v_h,w_h,\mu_h)] = m(u_0,v_h),\quad \forall
v_h,\,r_h,\,\mu_h \in   \mathcal{V}_h^{SR},
\end{equation*}
with
\begin{equation*}\begin{array}{ll}
A_{SR}[(u_h,q_h,\lambda_h),(v_h,w_h,\mu_h)] := & m(u_h,v_h)+s_1(u_h,v_h) \\[2mm]& +a_h(v_h,\lambda_h) + 
s_5(q_h,w_h) - m(\lambda_h,w_h) \\[2mm]& - a_h(u_h,\mu_h) +m(q_h,\mu_h).
\end{array}\end{equation*}
\section{Preliminary technical results}
First we define the semi norms associated with the stabilization
operator $s_i(\cdot,\cdot)$,
\[
|x_h|_{s_i} = s(x_h,x_h)^{\frac12}, \quad \forall x_h\in V_h.
\]
We recall the following  well known inverse and trace inequalities 
\begin{equation}\label{trace}
\begin{array}{rcl}
\|v\|_{L^2(\partial  K) } \leq C_t (h^{-\frac12} \|v\|_{ K } +
h^{\frac12} \|\nabla v\|_{ K }), \quad \forall v \in H^1( K ) \\[2mm]
h_{ K }\|\nabla v_h\|_{ K } + h_ K ^{\frac12} \|v_h\|_{L^2(\partial  K) }
\leq C_i \|v_h\|_ K , \quad \forall v_h \in V_h.
\end{array}
\end{equation}
As an immediate consequence of \eref{trace} we have the following
stabilities for some $C_{si}>0$ depending only on the mesh geometry,
\begin{equation}\label{eq:s_stab}
|x_h|_{s_i} \leq C_{si} \|h^{\frac{i-3}{2}} x_h\|_{L^2(\Omega)},\quad
|x_h|_{s_i} \leq C_{si} \|h^{\frac{i-1}{2}} \nabla x_h\|_{L^2(\Omega)}.
\end{equation}
Let $i_h:H^2(\Omega)\rightarrow V_h$ denote the standard Lagrange
interpolant and $\pi_h:L^2 \rightarrow V_h$ and $\pi^0_h:L^2 \rightarrow
V^0_h$ the $L^2$-projections on the respective finite element spaces. The
following error estimate is known to hold both for $i_h$ and $\pi_h$,
\begin{equation}\label{eq:approx_standard}
\|u - i_h u\|_{L^2(\Omega)}    +h \|\nabla (u - i_h u)\|_{L^2(\Omega)}
\leq C h^{t} |u|_{H^t(\Omega)}, \, t=1,2.
\end{equation}

To prove stability of our formulations below we need to show that for any
function $v_h \in V_h$ the $L^2$-norm is equivalent to
$
\|\pi^0_h v_h\|_{L^2(\Omega)} + |v_h|_{s_3}.
$
We prove this result in this technical Lemma.
\begin{lemma}\label{lem:stab_bound}
There exists $C_p>0$ such that for all $v_h \in V_h$ there holds
\begin{equation}\label{Poincare}
h \|v_h\|_{H^1(\Omega)} \leq C_p(\|v_h\|_{L^2(\mathcal{M})} + |v_h|_{s_1}).
\end{equation}
There exists $c_1, c_2 >0$ such that for any function $v_h \in V_h$ 
\begin{equation}\label{norm_eq}
c_1 (\|\pi^0_h v_h\|_{L^2(\Omega)} + |v_h|_{s_3}) \leq
\|v_h\|_{L^2(\Omega)} \leq c_2 (\|\pi^0_h v_h\|_{L^2(\Omega)} + |v_h|_{s_3}).
\end{equation}
\end{lemma}
\begin{proof}
The discrete Poincar\'e type inequality \eref{Poincare} may be proved
using a compactness argument similar to that of \cite{BHL15}. To keep
down the technical detail we here instead use an approach with
continuous Poincar\'e inequalities and discrete interpolation. Let $\mathcal{I}_h:\nabla V_h \mapsto
[V_h]^d$ be a quasi-interpolation operator \cite{BE07} such that
\begin{equation}\label{eq:quasi_prop}
\|\nabla v_h - \mathcal{I}_h \nabla v_h\|_{L^2(\Omega)} \leq C
|v_h|_{s_1}, \quad \| \mathcal{I}_h \nabla v_h\|_{L^2(K)} \leq C \|\nabla
v_h\|_{L^2(\Delta_K)} 
\end{equation}
where $\Delta_K := \cup_{K':K\cap K \ne \emptyset}$.
The following Poincar\'e inequality is well known (see \cite[Lemma B.63]{EG}). If
$f:H^1(\Omega) \mapsto \mathbb{R}$ is a linear functional that is
non-zero for constant functions then
\[
\|u\|_{H^1(\Omega)} \leq C_P ( |f(u)| + \|\nabla u\|_{L^2(\Omega)}),\quad
\forall u \in H^1(\Omega).
\]
For instance, we may take
\[
f(u) = \int_{\mathcal{M}} u ~\mbox{d}\Omega \leq C \|u\|_{L^2(\mathcal{M})}.
\]
As an immediate consequence we have the bound
\begin{equation}\label{eq:lemma2-a}
\|v_h\|_{H^1(\Omega)} \leq C (\|v_h\|_{L^2(\mathcal{M})} + \|\nabla v_h\|_{L^2(\Omega)}).
\end{equation}
Now let $\mathcal{M}_{int} \subset \mathcal{M}$, be the set
$\mathcal{M}_{int}:=\{K \subset \mathcal{M},\, \partial K \cap \partial \mathcal{M} =
\emptyset\}$, that is the set of interior triangles of
$\mathcal{M}$. It then follows by the stability of $\mathcal{I}_h$ that 
$\|\mathcal{I}_h \nabla v_h\|_{L^2(\mathcal{M}_{int})} \leq C \|\nabla
v_h\|_{L^2(\mathcal{M})}$. Adding and subtracting $\mathcal{I}_h \nabla v$ 
in the second term on the right hand side of (\ref{eq:lemma2-a}) gives
\begin{eqnarray}\nonumber
\|v_h\|_{H^1(\Omega)} &\leq C (\|v_h\|_{L^2(\mathcal{M})} 
+ \|\nabla v_h - \mathcal{I}_h \nabla v_h\|_{L^2(\Omega)} 
+ \|\mathcal{I}_h \nabla v_h\|_{L^2(\Omega)} )
\\ \label{eq:lemma2-b}
  &\leq C (\|v_h\|_{L^2(\mathcal{M})} 
+ | v_h |_{s_1}
+ \|\mathcal{I}_h \nabla v_h\|_{L^2(\Omega)} )
\end{eqnarray}
where we used (\ref{eq:quasi_prop}) in the second inequality. For the third term on the right hand side 
of (\ref{eq:lemma2-b}) we once again use Poincar\'e's inequality, the stability 
of $\mathcal{I}_h$ and the inverse inequality \eref{trace} to conclude that
\begin{eqnarray*}
\|\mathcal{I}_h \nabla v_h\|_{L^2(\Omega)} &\leq C(\|\mathcal{I}_h \nabla
v_h\|_{L^2(\mathcal{M}_{int})} 
+ \|\nabla \mathcal{I}_h \nabla v_h\|_{L^2(\Omega)})
\\
&\leq C (h^{-1} \|v_h\|_{L^2(\mathcal{M})} + \|\nabla \mathcal{I}_h \nabla
v_h\|_{L^2(\Omega)})
\end{eqnarray*} 
Using the fact that $\nabla v_h$ is constant on each element, so that
$$\|\nabla \mathcal{I}_h \nabla v_h \|^2_{L^2(\Omega)} = \sum_{K \in
  \mathcal{T}_h} \|\nabla (\nabla v_h - \mathcal{I}_h \nabla v_h)
  \|^2_{L^2(K)} 
  \leq C h^{-2} \|\nabla v_h - \mathcal{I}_h \nabla v_h
  \|^2_{L^2(\Omega)}
\leq C h^{-2} | v_h |_{s_1}^2  
  $$ where in the second inequality we have applied the
inverse inequality \eref{trace} to each term in the sum, and 
finally used \eref{eq:quasi_prop}. We thus obtain 
\begin{equation*}
\|\mathcal{I}_h \nabla v_h\|_{L^2(\Omega)}
\leq C h^{-1} ( \|v_h\|_{L^2(\mathcal{M})} + | v_h |_{s_1} )
\end{equation*}
which combined with (\ref{eq:lemma2-b}) gives 
\begin{eqnarray*}
h \|v_h\|_{H^1(\Omega)} 
&\leq  C (h + 1) (\|v_h\|_{L^2(\mathcal{M})} +  |v_h|_{s_1})
\leq
 C (\|v_h\|_{L^2(\mathcal{M})} +  |v_h|_{s_1}).
\end{eqnarray*}
By which we have proven \eref{Poincare}.

The lower bound of \eref{norm_eq} is immediate by the stability of the $L^2$-projection
and the inverse and trace inequalities of equation \eref{trace}. 
To prove the upper bound we write
\[
\|v_h\|^2_{L^2(\Omega)} = \|\pi^0_h v_h \|^2_{L^2(\Omega)} + \|v_h-\pi^0_h v_h \|^2_{L^2(\Omega)} 
\]
and let $w_h := v_h-\pi^0_h v_h$. We will now prove that
\begin{equation}\label{eq:key}
\|w_h\|_{L^2(\Omega)}^2 \leq |w_h|_{s_3}^2
\end{equation}
from which the upper bound follows, since by the triangle inequality 
followed by the first inequality of \eref{eq:s_stab}, with $i=3$, 
\[
|w_h|_{s_3} \leq |v_h|_{s_3}+ |\pi^0_h v_h|_{s_3} \leq
|v_h|_{s_3}+C_{si}  \|\pi^0_h v_h\|_{L^2(\Omega)}.
\]
To prove \eref{eq:key} we first define the support of a nodal basisfunctions
$\varphi_i$ by $\Omega_i := \{x \in \Omega:  \varphi_i(x) > 0\}$.
Then let $\mathcal{N}_I$ denote the set of indices of basis functions
that are in the interior of the domain, that is, for each value $i \in \mathcal{N}_I$ the closure of the support of the
associated basis function has empty intersection with the boundary,
$\overline \Omega_i \cap \partial \Omega = \emptyset$. For each $i \in
\mathcal{N}_I$ we define the macropatch $\Delta_i := \cup_{j: \Omega_j
  \cap \Omega_i \ne \emptyset} \Omega_j$. This means that $\Delta_i$
consist of $\Omega_i$ and any other patch $\Omega_j$ sharing two
triangles (in 2D) or
several tetrahedra (in 3D) with $\Omega_i$. Since $\mathcal{T}_h$ is shape
regular we may map the patch $\Delta_i$ to a shape regular $\tilde \Delta_i$ such that
$\mbox{diam}(\tilde \Delta_i) = 1$. We define the linear map $
F:\Delta_i \mapsto \tilde \Delta_i$. Let 
$\tilde{\mathcal{F}}$ and $\tilde{\mathcal{N}}$ denote the
set of interior faces and interior nodes respectively of $\tilde
\Delta_i$ and define the scalar product on $\tilde
\Delta_i$
\[
(\tilde v_h,\tilde y_h)_{\tilde \Delta_i} := \int_{\tilde \Delta_i}
\tilde v_h \tilde y_h |F^{-1}| \mbox{d} \tilde x.
\]
It follows that if $\tilde w_h$ denotes the mapped function $w_h$ and
$\tilde \varphi_j$ denotes the mapped basis function $\varphi_j$ then
\begin{equation}\label{eq:ortho_cond}
(\tilde w_h,\tilde \varphi_j)_{\tilde \Delta_i} = 0
\end{equation}
for all $j \in \tilde{\mathcal{N}}$. Now define the semi-norm $|\cdot
|_{s,{\tilde \Delta}}$ as
\[
|\tilde v_h
|_{s,{\tilde \Delta}}^2 := \sum_{F \in {\tilde{\mathcal{F}}}} \int_F
\jump{\tilde \nabla \tilde v_h \cdot n_F}^2 ~\mbox{d}\tilde x.
\]
We now prove that $|\cdot
|_{s, \tilde \Delta}$ is a norm on $\tilde w_h\vert_{\tilde \Delta_i}$. It is clearly a semi norm so
we only need to prove that $|\tilde w_h|_{s,{\tilde \Delta}} = 0$ implies
$\tilde w_h\vert_{\tilde \Delta_i} = 0$. If $|\tilde w_h|_{s,{\tilde \Delta}} = 0$
then $\tilde w_h$ is an affine function on $\tilde \Delta_i$. It is
straightforward to check that the only affine function that can
satisfy \eref{eq:ortho_cond} is the zero function. To be precise, otherwise $\tilde w_h$ has
to be odd with respect to the center of mass of all the basis
functions. Since $w_h$ is affine it will vanish along a line in 2D and
on a plane in 3D. This means that the centers of mass of all the
nodal basis functions associated to the nodes in $\tilde{\mathcal{N}}$
must be on the line (in the plane), which is impossible. The constant
of the estimate
depends on the shape regularity. Defining
\[
| v_h
|_{s_3,\Delta_i}^2 := \sum_{F \in {\mathcal{F}_i}} \int_F
h^3 \jump{ \nabla  v_h \cdot n_F}^2 ~\mbox{d} x,
\]
where $\mathcal{F}_i$ denotes the set of interior faces of $\Delta_i$,
we obtain, by scaling back to the physical geometry, that there
exists $C>0$ depending only on the local mesh geometry such that,
\[
\|w_h\|^2_{L^2(\Delta_i)} \leq C | w_h |_{s_3,\Delta_i}^2
\]
To conclude we observe that since the overlap between different
patches $\Delta_i$ is bounded uniformly in $h$ there holds
\[
\|w_h\|^2_{L^2(\Omega)} \leq \sum_{i \in \mathcal{N}_I}
\|w_h\|^2_{L^2(\Delta_i)} \leq   C \sum_{i \in \mathcal{N}_I}  | w_h
|_{s_3,\Delta_i}^2 \leq C |w_h|_{s_3}^2,
\]
which is the desired inequality.
\end{proof}
\section{Error analysis - data assimilation}

We introduce the triple norm
\begin{equation}\begin{array}{ll}\label{eq:tnorm_da}
\tnorm{(v_h,\mu_h)}:= & \|v_h\|_{L^2(\mathcal{M})} +\|h v_h\|_{H^1(\Omega)} + |v_h|_{s_1}+\|\mu_h\|_{H^1(\Omega)}
\end{array}\end{equation}
where
\[
|x_h|_{s_1} := s_1(x_h,x_h)^{\frac12}.
\]
Using \eref{eq:approx_standard} and \eref{trace} the  following approximation estimate is straightforward to show
\begin{equation}\label{approx_da}
\tnorm{(v-i_h v,0)} \leq C h
\|v\|_{H^2(\Omega)} .
\end{equation}
Observe that the terms in the above norm do not have matching
dimensions. Indeed there is a constant of the dimension of an inverse
length scale present in the first two terms in the right hand
side. In the term over $L^2(\mathcal{M})$ this is to avoid a too
strong penalty on possibly perturbed data and in the second term of
the right hand side it comes from the application of the discrete
Poincar\'e inequality \eref{Poincare}.

Stability in the norm
\eref{eq:tnorm_da} is sufficient to deduce the existence of a discrete
solution to the system \eref{eq:da1}-\eref{eq:da2}, however the norm is too weak
to be useful for error estimates.

The analysis takes the following form, following the framework of
\cite{Bu15}. First we prove inf-sup stability of the form
$A_{DA}[\cdot,\cdot]$ in the norm \eref{eq:tnorm_da}. From this the
existence of discrete solution to the linear system follows. Then we
show an error estimate in the norm \eref{eq:tnorm_da} that is
independent of the stability of the data assimilation problem and
gives convergence rates for the residuals of the
approximation. Finally we show that the error satisfies an equation of
the type \eref{Poisson}, with the right hand side given by the
residual. The a priori error estimates on the residual together with
the assumed a priori estimate on the exact solution allows us to
deduce error bounds through Lemma \ref{3sphere}.
\begin{proposition}\label{infsup_da}
Let $(u_h,\lambda_h) \in \mathcal{V}^{DA}_h$ be the solution of
\eref{eq:da1}-\eref{eq:da2} then there exists $c_s>0$ such that
\[
c_s \tnorm{(u_h,\lambda_h)} \leq \sup_{(v_h,\mu_h)  \in
  \mathcal{V}^{DA}_h} \frac{A_{DA}[(u_h,\lambda_h),(v_h,\mu_h)]}{\tnorm{(v_h,\mu_h)}}.
\]
\end{proposition}
\begin{proof}
First observe that 
\begin{eqnarray*}
A_{DA}[(u_h,\lambda_h),(u_h + \alpha \lambda_h , -\lambda_h)] &=
\|u_h\|^2_{L^2(\mathcal{M})} + |u_h|_{s_1}^2 + \alpha \|\nabla
\lambda_h\|_{L^2(\Omega)}^2 \\[2mm] & + \alpha m_{\mathcal{M}}(u_h, \lambda_h)+ \alpha s_1(u_h, \lambda_h).
\end{eqnarray*}
Using the Cauchy-Schwarz inequality, an arithmetic-geometric
inequality and the inverse inequality
\eref{trace} in the last term of the right hand side we get
\[
\alpha s_1(u_h, \lambda_h) \leq \frac12  |u_h|_{s_1}^2 + \frac12
\alpha^2 \gamma C_i^2 \|\nabla \lambda_h\|_{L^2(\Omega)}^2.
\]
and similarly 
\[
\alpha m_{\mathcal{M}}(u_h, \lambda_h) \leq \frac12  \|u_h\|_{L^2(\mathcal{M})}^2 + \frac12
\alpha^2 C_p^2 \|\nabla \lambda_h\|_{L^2(\Omega)}^2.
\]
Let $\alpha =  \frac12 \min(C_p^{-2},\gamma^{-1} C_i^{-2})$ to obtain
\[
A_{DA}[(u_h,\lambda_h),(u_h + \alpha \lambda_h , -\lambda_h)] \ge
\frac12 \|u_h\|^2_{L^2(\mathcal{M})} + \frac12 |u_h|_{s_1}^2 + \frac{\alpha}{2} \|\nabla
\lambda\|_{L^2(\Omega)}^2.
\]
The contribution $\|h u_h\|_{H^1(\Omega)}$ is added to the
right hand side by applying \eref{Poincare}.
Using once again the Cauchy-Schwarz inequality, an arithmetic-geometric
inequality and the inverse inequality
\eref{trace} we see that
\[
\tnorm{(u_h + \alpha \lambda_h , -\lambda_h)} \leq C \tnorm{(u_h , \lambda_h)},
\]
which, together with the Poincar\'e inequality for $\lambda_h$,
concludes the proof.
\end{proof}
\begin{proposition}\label{tnorm_error_da}
Let $(u_h,\lambda_h) \in \mathcal{V}^{DA}_h$ be the solution of
\eref{eq:da1}-\eref{eq:da2} and $u \in H^2(\Omega)$ the solution to
\eref{eq:minim}-\eref{Poisson}, with $\alpha=0$. Then
\[
\tnorm{(u-u_h,\lambda_h)} \leq C(\|\delta u_0\|_{L^2(\mathcal{M})} + \|\delta
  f\|_{H^{-1}(\Omega)}+ h |u|_{H^2(\Omega}).
\]
\end{proposition}
\begin{proof}
Let $\xi_h := u_h - i_h u$, with $i_h u$ denoting the Lagrange
interpolant of $u$. By the triangle inequality we have
\[
\tnorm{(u-u_h,\lambda_h)} \leq \tnorm{(u-i_h u,0)}+ \tnorm{(\xi_h,z_h)}
\leq C h |u|_{H^2(\Omega} + \tnorm{(\xi_h,\lambda_h)}.
\]
For the second term on the right hand side we apply the inf-sup
condition of Proposition \ref{infsup_da},
\begin{equation}\label{pert_infsup}
c_s \tnorm{(\xi_h,\lambda_h)} \leq \sup_{(v_h,\mu_h)  \in
  \mathcal{V}^{DA}_h} \frac{A_{DA}[(\xi_h,\lambda_h),(v_h,\mu_h)]}{\tnorm{(v_h,\mu_h)}}.
\end{equation}
Observing that $\forall
(v_h,\mu_h) \in   \mathcal{V}^{DA}_h$,
\begin{equation*}
A_{DA}[(u-u_h,\lambda_h),(v_h,\mu_h)] = m_{\Omega}(f-\tilde f, \mu_h ) + m_{\mathcal{M}}(u_0-\tilde u_0,v_h)-s_1(u_h,v_h),
\end{equation*}
we obtain the equality
\[
A_{DA}[(\xi_h,\lambda_h),(v_h,\mu_h)] = -m_{\mathcal{M}}(\delta u_0,
                                        v_h) - m_\Omega(\delta
                                        f,\mu_h) + a(u-i_hu,\mu_h)
                                        +s_1(i_h u,v_h)
\]
Using the Cauchy-Schwarz inequality in the terms of the right hand
side we immediately deduce
\begin{eqnarray*}
A_{DA}[(\xi_h,\lambda_h),(v_h,\mu_h)]&\leq& ( \|\delta u_0\|_{L^2(\mathcal{M})}+ \|\delta
  f\|_{H^{-1}(\Omega)} + \|\nabla
                                       (u-i_hu)\|_{L^2(\Omega)}+\tnorm{(u-i_hu,0)})
  \\
& \times& \tnorm{(v_h,\mu_h)}.
\end{eqnarray*}
Applying this inequality in \eref{pert_infsup} obtain the bound
\[
c_s \tnorm{(\xi_h,\lambda_h)} \leq \|\delta u_0\|_{L^2(\mathcal{M})} + \|\delta
  f\|_{H^{-1}(\Omega)}+\|\nabla (u-i_hu)\|_{L^2(\Omega)}+ \tnorm{(u-i_hu,0)}
\]
and the result follows from approximation estimate \eref{eq:approx_standard} and \eref{approx_da}.
\end{proof}

\begin{theorem}\label{L2_error_da}
Let $(u_h,\lambda_h) \in \mathcal{V}^{DA}_h$ be the solution of
\eref{eq:da1}-\eref{eq:da2} and $u \in H^2(\Omega)$ the solution to
\eref{eq:minim}-\eref{Poisson}, with $\alpha=0$. Then for some
$0<\tau<1$ depending on $r_1/r_3$, $r_2/r_3$ and the smallest
eigenvalue there holds
\begin{eqnarray*}
\|u - u_h\|_{L^2(B_{r_2}(\bfx_0))} & \leq& C_{h,\delta} (\|\delta u_0\|_{L^2(\mathcal{M})} + \|\delta
  f\|_{H^{-1}(\Omega)}+ |u_h|_{s_1})^\tau \\[2mm]
& \leq& C_{h,\delta} (\|\delta u_0\|_{L^2(\mathcal{M})} + \|\delta
  f\|_{H^{-1}(\Omega)} + h\|f\|_{L^2(\Omega)}+h |u|_{H^2(\Omega)})^{\tau}
\end{eqnarray*}
where
\begin{eqnarray*}
C_{h,\delta} &\leq& C
\|u-u_h\|_{L^2(\Omega)}^{(1-\tau)} \\[2mm]
& \leq& C (h^{-1}\|\delta u_0\|_{L^2(\mathcal{M})} +  h^{-1}\|\delta f\|_{H^{-1}(\Omega)}+  |u|_{H^2(\Omega)})^{(1-\tau)}.
\end{eqnarray*}
\end{theorem}
\begin{proof}
Let $e = u-u_h$. Then we have
\[
a(e,w) = (f,w) - a(u_h,w)=:\left< r,w \right>_{H^{-1},H^1}, \mbox{ with } r \in H^{-1}(\Omega).
\]
It follows that $e$ is a weak solution to the problem \eref{Poisson}
with the right hand side $r \in H^{-1}(\Omega)$ and we are in the
framework of Lemma \ref{3sphere}. We now only need to show the bound
\[
\|r\|_{H^{-1}(\Omega)} \leq \varepsilon
\]
and then apply the inequality \eref{3sphere_stab} with $e$ in the
place of $u$. By definition of the dual norm we have 
\[
\|r\|_{H^{-1}(\Omega)} = \sup_{w \in H^1_0(\Omega) \setminus 0 }\frac{\left<r,w\right>_{H^{-1},H^1}}{\|w\|_{H^1(\Omega)}}.
\]
We proceed using the definition of $r$, 
\begin{eqnarray*}
\left< r,w \right>_{H^{-1},H^1} & =  (f,w - i_h w) - a(u_h,w - i_h w) -
(\delta f, i_h w) \\[2mm]
& \leq C_r (h \|f\|_{L^2(\Omega)} + |u_h|_{s_1} +
\|\delta f\|_{H^{-1}(\Omega)}) \|w\|_{H^1(\Omega)}.
\end{eqnarray*}
Here we used partial integration in the form $a(\cdot,\cdot)$, a trace
inequality and approximation to
obtain
\[
|a(u_h,w - i_h w)| \leq \sum_{F \in \mathcal{F}_i} \int_{F}
|\jump{\nabla u_h \cdot n_F} |w - i_h w| ~\mbox{d}s \leq   C |u_h|_{s_1}\|w\|_{H^1(\Omega)}.
\]
We now apply Lemma \ref{3sphere} to obtain
\[
\|e\|_{B_{r_2}(x_0)} \leq C \left(\|e\|_{B_{r_1}(x_0)}+ \varepsilon\right)^\tau \cdot \left(\|e\|_{B_{r_3}(x_0)}+ \varepsilon\right)^{(1-\tau)}
\]
Setting $\varepsilon = C_r (h \|f\|_{L^2(\Omega)} + |u_h|_{s_1} +
\|\delta f\|_{H^{-1}(\Omega)})$ and observing that by Proposition \ref{tnorm_error_da}
\[
\|e\|_{B_{r_3}(x_0)}  \leq C h^{-1} (\|\delta u_0\|_{L^2(\mathcal{M})} + \|\delta
  f\|_{H^{-1}(\Omega)}+ h \|u\|_{H^2(\Omega})
\]
and once again by Proposition \ref{tnorm_error_da}
\[
\|e\|_{B_{r_1}(x_0)} = \|u - u_h\|_{L^2(\mathcal{M})} \leq C (\|\delta u_0\|_{L^2(\mathcal{M})} + \|\delta
  f\|_{H^{-1}(\Omega)}+ h \|u\|_{H^2(\Omega})
\]
we conclude.
\end{proof}
\begin{corollary}
Assume that for $h_0>0$ there holds
\begin{equation}\label{pert_bound}
\|\delta u_0\|_{L^2(\mathcal{M})} + \|\delta
  f\|_{H^{-1}(\Omega)} \leq h_0 \|u\|_{H^2(\Omega}
\end{equation}
then for $h > h_0$, there exists $C_0$ such that
\[
\|u - u_h\|_{L^2(B_{r_2}(\bfx_0))} \leq C_0 (h \|u\|_{H^2(\Omega)})^{\tau}
\]
with $C$ independent of $h$.
\end{corollary}
\begin{proof}
This result follows immediately by applying the assumed bound
\eref{pert_bound} in the error estimate of Theorem \ref{L2_error_da}.
\end{proof}
Observe that in particular the above result implies that if the exact
data is available in $\mathcal{M}$ we can compute the solution in
$B_{r_2}({\bf{x}}_0)$ to arbitrary precision. Also observe that
Theorem \ref{L2_error_da} provides both a priori and a posteriori
error bounds. This means that perturbations in data can be compared with the
computational residual in an a posteriori procedure to drive adaptive
algorithms for the computation of the reconstruction.
\section{Error analysis - source reconstruction}
The error analysis in this case follows a similar outline, however
instead of Lemma \ref{3sphere} we may here use a compactness argument
to obtain convergence orders.

We introduce the triple norm
\begin{equation}\label{eq:tnorm}
\tnorm{(v_h,w_h,\mu_h)}:= \|v_h\|_{L^2(\Omega)}+ \|h w_h\|_{L^2(\Omega)} +\|h^{-1}
\mu_h\|_{L^2(\Omega)}
+|v_h|_{s_1}+|w_h|_{s_5},
\end{equation}
where we recall that
\[
|x_h|_{s_i} := s_i(x_h,x_h)^{\frac12}.
\]
Using \eref{eq:approx_standard} and \eref{trace} the  following approximation estimate is straightforward to show
\begin{equation}\label{approx}
\tnorm{(v-i_h v, w - \pi_h w,0)} \leq C h
(\|v\|_{H^2(\Omega)} + \|w\|_{L^2(\Omega)}).
\end{equation}
We will also use the Ritz-projection defined by $r_h u \in V_h^0$ such
that
\[
a_h(r_hu,v_h) = a_h(u,v_h),\quad \forall v_h \in V_h^0.
\]
It is well known that if $\Omega$ is convex the following estimate
holds
\[
\|u - r_hu\|_{L^2(\Omega)} + h \|\nabla (u - r_hu) \|_{L^2(\Omega)}
\leq C h^2 |u|_{H^2(\Omega)}.
\]
We will first prove an estimate where we assume that $q$ is more
regular
and show that in this case the stabilization of the velocity is
superfluous.
\begin{proposition}
Let $(u,q) \in W\times H^1(\Omega)$ satisfy \eref{eq:constraint} and
let $(u_h,q_h\lambda_h) \in \mathcal{V}_h^{SR}$ be the solution
of \eref{eq:crazy1}-\eref{eq:crazy2}, with $\gamma_1=0$ and
$\gamma_5\ge 0$. Then there holds
\begin{equation}\label{bound1}
\|u - u_h\|+h \|\nabla (u-u_h)\|_{L^2(\Omega)} + |q - q_h|_{s_5}\leq
C (h^2 \gamma_5^{\frac12} \| q\|_{H^1(\Omega)} +  \|u - u_0\|_{L^2(\Omega)})
\end{equation}
For $\gamma_5 \ge 0$
\begin{equation}\label{bound2}
\|\pi_h^0(q-q_h)\|_{H^{-1}(\Omega)} \leq
C (h  \| q\|_{H^1(\Omega)} + h^{-1} \|u - u_0\|_{L^2(\Omega)})
\end{equation}
and for $\gamma_5>0$
\begin{equation}\label{bound3}
\|q-q_h\|_{H^{-1}(\Omega)} \leq
C (h  \| q\|_{H^1(\Omega)} + h^{-1} \|u - u_0\|_{L^2(\Omega)})
\end{equation}
\end{proposition}
\begin{proof}
Let $\xi_h = u_h - r_h u$ and $\eta_h = q_h - \pi_h q$. It follows by
the definition of $A_h[(\cdot,\cdot,\cdot),(\cdot,\cdot,\cdot)]$ that
\[
\|\xi_h\|_{L^2(\Omega)}^2+ |\eta_h|_{s_5}^2 =  A_h[(\xi_h,\eta_h,\lambda_h),(\xi_h,\eta_h,-\lambda_h)].
\]
Using that
\[
A_h[(u - u_h,q-q_h,\lambda_h),(\xi_h,\eta_h,-\lambda_h)] = m(u - u_0,\xi_h) 
\]
we have
\[
\|\xi_h\|_{L^2(\Omega)}^2+ |\eta_h|_{s_5}^2 = (u - r_h u, \xi_h) +
s_5(\pi_h q, \eta_h)- m(u - u_0,\xi_h).
\]
Using a Cauchy-Schwarz inequality we then obtain
\[
\|\xi_h\|_{L^2(\Omega)}+ |\eta_h|_{s_5} \leq \|u - r_h u\|_{L^2(\Omega)} +
|\pi_h q|_{s_5}+ \|u - u_0\|_{L^2(\Omega)}.
\]
Using the approximation properties of the Ritz projection and the
$H^1$-stability of the $L^2$-projection we get the estimate
\[
\|\xi_h\|_{L^2(\Omega)}+ |\eta_h|_{s_5} \leq C h^2(|u|_{H^2(\Omega)} +
\|\nabla q\|_{L^2(\Omega)}) + \|u - u_0\|_{L^2(\Omega)}.
\]
The estimate on the gradient of the error is then a consequence of an
inverse inequality
\[
\|\nabla \xi_h\|_{L^2(\Omega)} \leq C_i h^{-1} \|\xi_h\|_{L^2(\Omega)}
\leq C h (|u|_{H^2(\Omega)} +
\|\nabla q\|_{L^2(\Omega)}) + h^{-1} \|u - u_0\|_{L^2(\Omega)}.
\]
For the estimate on the source term observe that for $\gamma_5 \ge 0$
\begin{eqnarray}\label{pi0bound}
\|\pi_h^0(q-q_h)\|_{H^{-1}(\Omega)} & = &\sup_{w\in H^1(\Omega): \|w\|_{H^1} = 1}
(\pi_h^0(q - q_h), w - \pi^0_h w) \\[2mm] \nonumber
& & + a(u - u_h,\pi_h^0 w) \\[2mm]
&\leq & C (h \|q - \pi_h
q\|_{L^2(\Omega)} + \|\nabla (u - u_h)\|_{L^2(\Omega)}). \nonumber
\end{eqnarray}
If on the other hand $\gamma_5>0$ then we may use 
\[
\|q - q_h\|_{H^{-1}(\Omega)} \leq \|q - \pi_h q\|_{H^{-1}(\Omega)} +
\|\pi_h q - q_h\|_{H^{-1}(\Omega)},
\]
where it is immediate to show that $ \|q - \pi_h q\|_{H^{-1}(\Omega)}
\leq C h^2 \|\nabla q\|_{L^2(\Omega)}$ and 
\[
\|\pi_h q - q_h\|_{H^{-1}(\Omega)} \leq \|(\pi_h q - q_h)-\pi_h^0(\pi_h q - q_h)\|_{H^{-1}(\Omega)} +\|\pi_h^0(q - q_h)\|_{H^{-1}(\Omega)} .
\]
For the second term on the right hand side the estimate
\eref{pi0bound} holds and for the first term we observe that with
$\eta_h-\pi_h^0 \eta_h
= (\pi_h q - q_h)-\pi_h^0(\pi_h q - q_h)$ we have
\begin{eqnarray*}
\|\eta_h-\pi_h^0 \eta_h \|_{H^{-1}(\Omega)} &= \sup_{w\in H^1(\Omega): \|w\|_{H^1} = 1}
(\eta_h-\pi_h^0 \eta_h, w - \pi^0_h w) \\ [2mm]&\leq h \|\eta_h-\pi_h^0
\eta_h\|_{L^2(\Omega)} \leq C (|\eta_h|_{s_5} + h \|\pi_h^0 \eta_h\|_{L^2(\Omega)}. 
\end{eqnarray*}
We may then use the equation to deduce
\[
h^2 \|\pi_h^0 \eta_h\|^2_{L^2(\Omega)} = h^2 (q- q_h,\pi_h^0 \eta_h)=
h^2a(u-u_h,\pi_h^0 \eta_h)
\]
and after a Cauchy-Schwarz inequality and an inverse inequality in the
second factor,
\begin{eqnarray*}
h^2a(u-u_h,\pi_h^0 \eta_h)\leq C h \|\nabla(u-u_h)\|_{L^2(\Omega)} \|\pi_h^0 \eta_h\|_{L^2(\Omega)}.
\end{eqnarray*}
It follows that $h\|\pi_h^0 \eta_h\|_{L^2(\Omega)} \leq C h
\|\nabla(u-u_h)\|_{L^2(\Omega)}$ and the above bounds that
\eref{bound3} holds.
\end{proof}
We see that if less regularity is assumed for the source term $q \in
L^2(\Omega)$ then convergence of the gradient can no longer be
deduced, however a priori bounds on $u_h$ and $q_h$ are nevertheless
achieved that may be used to prove convergence in the asymptotic
limit. In order to obtain an estimate with convergence order in $h$
also in the case where $q\in L^2(\Omega)$ we take $\gamma_1>0$
and prove that this allows us to obtain stronger control of the
approximation of the source term. This in its turn allows us to prove
convergence
using an interpolation argument between $H^{-2}$ and $L^2$ as we shall
see below.

\begin{proposition}\label{infsup}(inf-sup stability)
For all $(y_h,t_h,\varsigma_h) \in \mathcal{V}_h^{SR}$ there
exists $c_s>0$ such that
\[
c_s \tnorm{(y_h,t_h,\varsigma_h)} \leq \sup_{(v_h,w_h,\mu_h) \in
  \mathcal{V}_h^{SR}} \frac{A_{SR}[(y_h,t_h,\varsigma_h),
  (v_h,w_h,\mu_h)]}{\tnorm{(v_h,w_h,\mu_h)}}.
\] 
\end{proposition}
\begin{proof}
For some $\beta>0$ to be fixed, take $v_h = y_h$, $w_h = t_h - \beta h^{-2} \varsigma_h$, $\mu_h=\varsigma_h +
\beta h^2 \pi^0_h t_h$ to obtain
\begin{equation*}\begin{array}{ll}
& A_{SR}[(y_h,t_h,\varsigma_h), (y_h, t_h - \beta  h^{-2} \varsigma_h,\varsigma_h +
\beta h^2  \pi^0_h t_h)]  =  |y_h|^2_{s_1} + |t_h|^2_{s_5} +
\|y_h\|^2_{L^2(\Omega)} \\[2mm] & \qquad  + 
\beta \|h^{-1} \varsigma_h\|^2_{L^2(\Omega)} 
+ \beta\|h  \pi^0_h t_h\|^2_{L^2(\Omega)} -
\beta s_5(t_h,h^{-2} \varsigma_h) - \beta  a_h(y_h, h^2 \pi^0_h t_h).
\end{array}\end{equation*}
Observing that Cauchy-Schwarz inequalities, arithmetic-geometric
inequalities and the stability inequality \eref{eq:s_stab}, with $i=5$ leads to
\[
s_5(t_h,h^{-2} \varsigma_h) \leq  \frac{1}{2} |t_h|^2_{s_5}+\frac12
C_{si}^2  \|h^{-1} \varsigma_h\|^2_{L^2(\Omega)}.
\]
Using partial integration, the fact that $ \pi^0_h t_h\vert_{\partial
  \Omega} = 0$ and a trace inequality \eref{trace} we have the bound
\[
a_h(y_h, h^2\pi^0_h  t_h) \leq \frac1{2} |y_h|_{s_1}^2 +\frac12 C^2_t
 \|h  \pi^0_h t_h\|^2_{L^2(\Omega)}.
\]

We may then fix $\beta = \min(C_{si}^{-2}, C_t^{-2})$ to show that for some $c>0$ depending only on the mesh geometry
\begin{equation*}\begin{array}{l}
c( |y_h|^2_{s_1} + |t_h|^2_{s_5} +
\|y_h\|^2_{L^2(\Omega)}+ 
\alpha \|h^{-1} \varsigma_h\|^2_{L^2(\Omega)} 
+ \alpha \|h  \pi^0_h t_h\|^2_{L^2(\Omega)}  ) \\[2mm]
\qquad\leq A_{SR}[(y_h,t_h,\varsigma_h), (y_h, t_h  - \beta h^{-2} \varsigma_h,\varsigma_h +
\beta h^2 t_h)].
\end{array}\end{equation*}
By Lemma \ref{lem:stab_bound} and the quasi uniformity of the mesh there holds for some $c>0$ depending
only on the mesh geometry
\[
c \|h  t_h\|^2_{L^2(\Omega)} \leq \|h  \pi^0_h t_h\|^2_{L^2(\Omega)} + |t_h|^2_{s_5} 
\]
and it follows that for some $c>0$ depending
only on the mesh geometry
\[
c \tnorm{(y_h,t_h,\varsigma_h)}^2 \leq A_{SR}[(y_h,t_h,\varsigma_h), (y_h, t_h + h^{-2} \varsigma_h,\varsigma_h +
h^2 t_h)].
\]
To conclude we need to show that
\[
\tnorm{ (y_h, t_h + h^{-2} \varsigma_h,\varsigma_h +
h^2 t_h)} \leq C \tnorm{(y_h,t_h,\varsigma_h)}.
\]
To this end we note that
\[
\tnorm{ (y_h, t_h + h^{-2} \varsigma_h,\varsigma_h +
h^2 t_h)} \leq \tnorm{(y_h,t_h,\varsigma_h)} + \tnorm{ (0, h^{-2} \varsigma_h,
h^2 t_h)}.
\]
For the second term in the right hand side we may write
\begin{equation*}\begin{array}{l}
\tnorm{ (0, h^{-2} \varsigma_h,
h^2 t_h)}^2 = \|h (h^{-2} \varsigma_h)\|_{L^2(\Omega)}^2 + \|h^{-1}
(h^2 t_h)\|_{L^2(\Omega)}^2 +|h^{-2} \varsigma_h|_{s_5}^2\\[2mm]
\qquad \leq C (\|h^{-1} \varsigma_h\|_{L^2(\Omega)}^2 + \|
h t_h\|_{L^2(\Omega)}^2 +|\varsigma_h|_{s_1}^2)
\leq C \tnorm{(y_h,t_h,\varsigma_h)}^2.
\end{array}\end{equation*}
where the constant only depends on the constant of quasiuniformity of
the meshes.
\end{proof}
\begin{proposition}\label{first_estimate}
Let $(u,q) \in W\times L^2(\Omega)$ satisfy \eref{eq:constraint} and
let $(u_h,q_h\lambda_h) \in \mathcal{V}_h^{SR}$ be the solution
of \eref{eq:crazy1}-\eref{eq:crazy2}. Then there holds
\[
\tnorm{(u-u_h, q-q_h,\lambda_h)} \leq C (h \|q\|_{L^2(\Omega)} + \|u - u_0\|_{L^2(\Omega)})
\]
and 
\[
\|\nabla (u - u_h)\|_{L^2(\Omega)} + \|q - q_h\|_{L^2(\Omega)} \leq C (\|q\|_{L^2(\Omega)} + h^{-1} \|u - u_0\|_{L^2(\Omega)}).
\]
\end{proposition}
\begin{proof}
Considering \eref{approx} it is sufficient to prove the estimate for
the discrete errors  $\xi_h:= u_h - i_h u$ and $\eta_h := q_h -\pi_h
q$. From Proposition \ref{infsup} we know
\[
c_s \tnorm{(\xi_h,\eta_h,\lambda_h)} \leq \sup_{(v_h,w_h,\mu_h) \in
  \mathcal{V}_h^{SR}} \frac{A_{SR}[(\xi_h,\eta_h,\lambda_h),
  (v_h,w_h,\mu_h)]}{\tnorm{(v_h,w_h,\mu_h)}}.
\] 
Using the definition of \eref{eq:constraint},
\eref{eq:crazy1}-\eref{eq:crazy2} and the orthogonality of the
$L^2$-projection we may write
\begin{equation*}\begin{array}{lll}
A_{SR}[(\xi_h,\eta_h,\lambda_h), (v_h,w_h,\mu_h)] & = & m(u_0-u,v_h) - m(i_h u - 
u,v_h) - s_1(i_h u,v_h) \\[2mm] && + s_5(\pi_h q, w_h)- a_h(u- i_h
u,\mu_h)\\[2mm] & = & I+II+III+IV+V.
\end{array}\end{equation*}
We see that using a Cauchy-Schwarz inequality, the approximation
properties of the Lagrange-interpolant and the regularity of $u$  we have
\[
I + II \leq (\|u - u_0\|_{L^2(\Omega)} + C h^2 \|q\|_{L^2(\Omega)}) ) \|v_h\|_{L^2(\Omega)}.
\]
Similarly for term three we use Cauchy-Schwarz inequality and the
approximation result \eref{approx} to obtain
\[
III \leq s_1(u - i_h u, u - i_h u)^{\frac12} s_1(v_h,v_h)^{\frac12}
\leq C h \|q\|_{L^2(\Omega)}  |v_h|_{s_1} .
\]
In term $IV$ we use an integration by parts followed by the stability
\eref{eq:s_stab}$_1$, with $i=5$, to obtain
\[
IV \leq s_5(\pi_h q,\pi_h q)^{\frac12} s_5(w_h,w_h)^{\frac12} \leq C h
\|q\|_{L^2(\Omega)} | w_h|_{s_5}.
\]
Finally for term $V$ we proceed using Cauchy-Schwarz inequality followed
by an inverse inequality to obtain
\begin{equation*}
V \leq \|\nabla (u - i_h u)\|_{L^2(\Omega)} \|h^{-1}
\mu_h\|_{L^2(\Omega)}  \leq  C h
\|q\|_{L^2(\Omega)}  \|h^{-1}
\mu_h\|_{L^2(\Omega)}.
\end{equation*}
Collecting the bounds for the terms $I-V$ above we deduce that the
following bound holds
\begin{eqnarray*}
&A_{SR}[(\xi_h,\eta_h,\lambda_h), (v_h,w_h,\mu_h)]  \leq C( h \|q\|_{L^2(\Omega)} + \|u
- u_0\|_{L^2(\Omega)}) \\[2mm]
&\qquad \qquad \times (\|v_h\|^2_{L^2(\Omega)}  + | w_h|^2_{s_5}  +
\|h^{-1} \mu_h\|_{L^2(\Omega)}^2 +  |v_h|^2_{s_1} )^{\frac12}\\[2mm]
&\qquad \leq C( h \|q\|_{L^2(\Omega)} + \|u
- u_0\|_{L^2(\Omega)}) \tnorm{(v_h,w_h,\mu_h)}
\end{eqnarray*}
which proves the first claim. The second claim follows immediately by
the definition of the triple norm, an inverse inequality and the first inequality.
\end{proof}
\begin{theorem}\label{th:perturb}
Let $(u,q) \in W \times L^2(\Omega)$ satisfy \eref{eq:constraint} and
let $(u_h,\lambda_h)$ be the solution
of \eref{eq:crazy1}-\eref{eq:crazy2}. Then there holds
\[
\|\nabla (u - u_h)\|_{L^2(\Omega)} + \|q - q_h\|_{H^{-1}(\Omega)} \leq
C( h^{\frac12} \|q\|_{L^2(\Omega)}+ h^{-\frac12} \|u - u_0\|_{L^2(\Omega)}).
\]
\end{theorem}
\begin{proof}
Define 
\begin{equation}\label{eq:neg2def}
\|q - q_h\|_{-2} := \sup_{y \in W :\; \|y\|_{H^2(\Omega)} = 1}
( q - q_h,y)_{L^2(\Omega)}.
\end{equation}
Then we may use the formulation to write for all $y \in W$ with $\|y\|_{H^2(\Omega)} = 1$,
\[
( q - q_h,y)_{L^2(\Omega)}  =  ( q - q_h,y - i_h y)_{L^2(\Omega)} +
a_h(u - u_h, i_h y) + s_1(\lambda_h, i_h y).
\]
After integration by parts in the second term, using that
$(u-u_h)\vert_{\partial \Omega}=0$, and using the first bound of
Proposition \ref{first_estimate} and the approximation bounds 
\[
\|h^{-\frac12} (u - i_h u)\|_{\mathcal{F}_I} \leq C h
\|u\|_{H^2(\Omega)} \leq C h \|q\|_{L^2(\Omega)} \mbox{ and }  s_1(i_h
y, i_h y)^{\frac12} \leq C h
\]
 we get
\begin{equation*}\begin{array}{lll}
( q - q_h,y)_{L^2(\Omega)} & \leq & C h^2 \|q - q_h\|_{L^2(\Omega)} + \sum_{K} \frac12 \int_{\partial K
  \setminus \partial \Omega} |u - u_h||\jump{\nabla i_h y\cdot n}|
~\mbox{d}s\\[2mm]
& \leq &
 C h^2 \|q - q_h\|_{L^2(\Omega)}  \\[2mm] && +(\|h^{-\frac12} (u - i_h u)\|_{\mathcal{F}_I} +
 \|h^{-1} \xi_h\|_{L^2(\Omega)}) s_1(i_h y, i_h y)^{\frac12} \\[2mm]
& \leq &
C (h \|q \|_{L^2(\Omega)} + \|u_0 - u\|_{L^2(\Omega)}).
\end{array}\end{equation*}
It follows that
\begin{equation}\label{min2est}
\|q - q_h\|_{-2} \leq C (h \|q \|_{L^2(\Omega)} + \|u_0 - u\|_{L^2(\Omega)}).
\end{equation}
Now consider the problem: find $z \in H^1_0(\Omega)$
such that
\[
-\Delta z = q - q_h \mbox{ in the sense of distributions}.
\]
This problem is well posed with $z$ that satisfies
$\|z\|_{H^2(\Omega)} \leq C \|q - q_h\|_{L^2(\Omega)}$.
The $H^{-1}$-norm of $q-q_h$ may be written
\[
\|q - q_h\|_{H^{-1}(\Omega)} := \sup_{y \in H_0^1(\Omega) :\; 
    \|y\|_{H^1(\Omega)} = 1} ( q - q_h,y)_{L^2(\Omega)}.
\]
For all $y \in H^1_0(\Omega)$ we may write
\[
 ( q - q_h,y)_{L^2(\Omega)} = (-\Delta z, y)_{L^2(\Omega)} = (\nabla
 z, \nabla y)_{L^2(\Omega)} \leq \|\nabla z\|_{L^2(\Omega)} \|\nabla y\|_{L^2(\Omega)}.
\]
It follows that
\[
\|q - q_h\|_{H^{-1}(\Omega)} \leq \|\nabla z\|_{L^2(\Omega)} = (\nabla
z, \nabla z)^{\frac12}_{L^2(\Omega)} = (-\Delta z, z) _{L^2(\Omega)}^{\frac12}.
\]
Considering the last term in the right hand side we see that
\begin{equation*}\begin{array}{lll}
 (-\Delta z, z) _{L^2(\Omega)}&=& (q -q_h, z)_{L^2(\Omega)} \leq \|q -
q_h\|_{-2} \|z\|_{H^2(\Omega)} \\[2mm] &\leq& C \|q -
q_h\|_{-2} \|q -
q_h\|_{L^2(\Omega)} \\[2mm]
&\leq &Ch^{-1} (h \|q \|_{L^2(\Omega)} + \|u_0 - u\|_{L^2(\Omega)})^2.
\end{array}\end{equation*}
Taking the square root of the last expression we obtain the desired estimate
for $\|q-q_h\|_{H^{-1}(\Omega)}$.
We now consider the estimate on the error in the gradient of
$u$. Since $u, u_h \in H^1_0(\Omega)$ there holds
\begin{equation*}\begin{array}{lll}
\|\nabla ( u - u_h)\|^2_{L^2(\Omega)} & = & a_h(u- u_h,u-u_h) \\[2mm]
& = &\underbrace{(q,u-u_h)_{L^2(\Omega)}}_{I} - \underbrace{ (q - q_h,
  u_h - u)_{L^2(\Omega)}}_{II} - \underbrace{ (q - q_h,
  u)_{L^2(\Omega)}}_{III}.
\end{array}\end{equation*}
For the first and second terms of the right hand side we use the first
estimate of Proposition \ref{first_estimate} to obtain
\begin{equation*}\begin{array}{lll}
I \leq \|q\|_{L^2(\Omega)} \|u -
u_h\|_{L^2(\Omega)}  & \leq & C \|q\|_{L^2(\Omega)} ( h \|q\|_{L^2(\Omega)} + \|u -
u_0\|_{L^2(\Omega)}) \\[2mm]
& \leq & C ( h \|q\|^2_{L^2(\Omega)} + h^{-1} \|u -
u_0\|^2_{L^2(\Omega)})
\end{array}\end{equation*}
and
\begin{equation*}\begin{array}{lll}
II & \leq & \|q - q_h\|_{L^2(\Omega)} \|u - u_h\|_{L^2(\Omega)} \\[2mm]
& \leq & C (\|q\|_{L^2(\Omega)} +
h^{-1} \|u - u_0\|_{L^2(\Omega)}) (h \|q\|_{L^2(\Omega)} +
\|u - u_0\|_{L^2(\Omega)})\\[2mm]
 & \leq & C ( h \|q\|^2_{L^2(\Omega)} + h^{-1} \|u -
u_0\|^2_{L^2(\Omega)}).
\end{array}\end{equation*}
For the third term we observe that by the definition
\eref{eq:neg2def}, the bound \eref{min2est} and the regularity assumption on $u$ we may write
\begin{equation*}\begin{array}{lll}
III & \leq & \|q - q_h\|_{-2} \|u\|_{H^2(\Omega)} \leq C( h \|q\|_{L^2(\Omega)}+ \|u_0 -
u\|_{L^2(\Omega)} \|) \|q\|_{L^2(\Omega)} \\[2mm]
& \leq &  C ( h \|q\|^2_{L^2(\Omega)} + h^{-1} \|u -
u_0\|^2_{L^2(\Omega)}).
\end{array}\end{equation*}
Collecting the bounds of the terms $I$-$IV$ and taking square roots
concludes the proof.
\end{proof}
\subsection{The effect of measurement error}
Both the above estimates take into account the measurement errors. We
will here discuss the second case in some more detail.
Observe that these estimates are nonstandard in the sense that they
measure the distance of the approximate solution to {\emph{any}} pair
of functions $(u,q)$ that satisfy \eref{eq:constraint}. First assume that
$u_0 \in W$ then $u_0$ satisfies the constraint and $u_h$ will
converge to $u_0$ and $q_h$ to $-\Delta u_0$, according to the
estimates above. Otherwise under our assumptions $u_0$ is a
measurement on the form $u_0 = \tilde u_0 +\delta u_0$ and in this
case the solution that we wish to approximate is $u = \tilde u_0$ and
$q=-\Delta \tilde u_0$. Provided $\|\delta u_0\|_{L^2(\Omega)}$ is a
priori known the estimates above give us a bound on
how close $u_h,q_h$ are to $\tilde u_0$ and
$-\Delta \tilde u_0$. The a perturbation term $\|\delta
u_0\|_{L^2(\Omega)}$ does not vanish under mesh-refinement and that
causes the estimates to blow up if $u_0$ is not in the range of the Laplace operator.

The above results also gives us quantitative information on how the perturbation $\delta u_0 $ in the
measurements affects the computation. For general measurement errors we see that formally refinement
improves the solution as long as
\[
\frac{\|\delta u_0\|_{L^2(\Omega)}}{\|q\|_{L^2(\Omega)}} {<}{<} C h.
\]
Since $q$ is unknown it has to be replaced by $q_h$ in practice, which
is be reasonable as long as $\|q_h\|_{L^2(\Omega)}$ does not grow
under mesh refinement.

If we assume that $u$ is sampled
on some
coarse scale $H$ and $u_0$ is taken to be a piecewise affine interpolant using these
samples. Then $\|u - u_0\|_{L^2(\Omega)} \leq C H^2 \|u\|_{H^2(\Omega)}$ and the
above estimate becomes
\[
\|\nabla (u - u_h)\|_{L^2(\Omega)} + \|q - q_h\|_{H^{-1}(\Omega)} \leq
C( h^{\frac12} + H^2 h^{-\frac12} ) \|q\|_{L^2(\Omega)}.
\]
We see that we have two terms with different behavior as we refine in
$h$. To ensure that the solution improves under mesh refinement we can 
check that the discretization error dominates the measurement
error. This is true as long as
$
h^{\frac12} >>  H^2 h^{-\frac12}
$
or $H^2 << h$. Another a posteriori criterion for when refining the mesh
improves the solution is given by:
\begin{enumerate}
\item  $\|\nabla u_h\|_{L^2(\Omega)}$ and $\|q_h\|_{L^2(\Omega)}$
non-increasing under refinement 
\item $s_1(u_h,u_h),\, s_5(q_h,q_h)$ decreasing
under refinement. 
\end{enumerate}
\section{Numerical examples}
\subsection{Data assimilation}
We consider problem \ref{eq:minim} on the domain $\Omega = (-1,1)\times(-1,1)$ and the 
data $$u_0=(x+1)^2(x-1)(y+1)(y-1)^2$$ with right--hand side $f=-\Delta u_0$, and $\mathcal{M}= (-1/4,1/4)\times(-1/4,1/4)$. We set $\gamma = 10^{-4}$.
The first mesh on $\Omega$, in a sequence, and the mesh on $\mathcal{M}$ is given in Fig. \ref{fig:meshes}. The different domains $B_{r_2}$ on which convergence is measured is 
shown in Fig. \ref{fig:Bdiff}, and the obtained convergence is shown in Fig. \ref{convassim}. Note that the error constant increases with $r_2$ and that the convergence rate decreases slowly, cf. Table \ref{table1}.
\begin{table}\label{table1}
\centering
\begin{tabular}{|c| c | c |c|}
\hline
No. of nodes  & Distance & $L_2$ error & Rate  
\\
\hline
\hline
1313     & 0 &  $0.76\times 10^{-3}$  & - 
\\
\hline
5185     & 0 & $0.53\times 10^{-3}$ & 0.52 
\\
\hline
20609     & 0 & $0.29\times 10^{-3}$ & 0.85
\\
\hline
82177     & 0 & $0.18\times 10^{-3}$ & 0.75
\\
\hline
\hline
1313  & 0.1875 & $0.81\times 10^{-2}$ & - 
\\
\hline
5185     & 0.1875 &  $0.61\times 10^{-2}$ & 0.41 
\\
\hline
20609    & 0.1875 & $0.41\times 10^{-2}$ & 0.58 
\\
\hline
82177     & 0.1875 & $0.29\times 10^{-2}$ & 0.50
\\
\hline
\hline
1313     & 0.375 & $0.37\times 10^{-1}$ & -
\\
\hline
5185     & 0.375 & $0.30\times 10^{-1}$ & 0.30
\\
\hline
20609     & 0.375 & $0.23\times 10^{-1}$ & 0.41
\\
\hline
82177     & 0.375 & $0.17\times 10^{-1}$ & 0.40
\\
\hline
\hline
1313     & 0.625 & 0.20 & -
\\
\hline
5185     & 0.625 & 0.19 & 0.11
\\
\hline
20609     & 0.625 & 0.15 & 0.28
\\
\hline
82177    & 0.625 & 0.13 & 0.30
\\
\hline
\end{tabular}
\caption{Errors and convergence rates underlying Figure \ref{fig:Bdiff}; meshsize is defined as the inverse square root of the number of nodes.}
\end{table}

\subsection{Source reconstruction}
\subsubsection{Convergence for smooth and non--smooth sources, with perturbation of data}

Again, we consider the domain $\Omega = (-1,1)\times(-1,1)$ and the 
data $u_0=(x+1)(x-1)(y+1)(y-1)$, corresponding to the smooth source term $q=2(2-x^2-y^2)$.

In Fig. \ref{tikjump} we show the interpolant of the exact source and a typical solution obtained using the gradient jump stabilization method. In Fig. \ref{tikh1} we show the effact of (properly scaled) Tikhonov regularization using $L_2$, i.e., $\| q\|$,  and $H^1$, i.e.,
$\| \nabla q\|$, regularizations, respectively. Note that the $L_2$ regularization gives the wrong boundary conditions in the discrete scheme, whereas $H^1$ works better while still giving a spurious boundary effect, well known from similar approaches used in fluid mechanics, cf. Burman and Hansbo \cite{BH06}.

The observed convergence using (\ref{eq:crazy1})--(\ref{eq:crazy2}) using only the stabilization term $s_5(q_h,w_h)$ and for a variety of choices of $\gamma$ is shown in Fig. \ref{smooth}.
We note that the convergence $\| q_h-q\|_{L_2(\Omega)}$ and $\| u_h-u\|_{H^1(\Omega)}$ is first order in both cases. The convergence of $u_h$ is completely unaffected by the choice of $\gamma$.

For the non--smooth case we let the solution be two different constant in the radial direction, $u=1$ for $0\leq r\leq 1/4$ and $u=0$ for $3/4\leq r$, interconnected by a cubic $C^1$-polynomial i the radial direction. This means that the source term will have jumps at $r=1/4$ and $r=3/4$ so that 
$q\in H^{1/2-\epsilon}(\Omega)$ for any $\epsilon >0$. In Fig. \ref{nonsmooth} we show the observed rate of convergence, which drops to about $O(h^{1/2})$ for $\| q_h-q\|_{L_2(\Omega)}$ but remains $O(h)$ for $\| u_h-u\|_{H^1(\Omega)}$. The error constant is now affected by $\gamma$ also for the convergence of $u_h$.

Finally, we show the effect of perturbing the data randomly, with constant amplitude and with amplitude decreasing $O(h)$. In Fig. \ref{perturbed} we show the obtained convergence in $H^1$--seminorm. The convergence is $O(h^{-1/2})$ and $O(h)$, respectively, cf. Theorem \ref{th:perturb}.

\subsubsection{Measurement error}

Consider $\Omega =
(0,1)\times (0,1)$ and the right hand side defined as a discontinuous
cross shaped function (see Figure \ref{plot}, left plot) written using boolean binary functions as
\[
f=(x>1/3)*(x<2/3)+ (y>1/3)*(y<2/3).
\]
The data $u_0$ is reconstructed using $P_4$ finite elements on the one
hand on a mesh
that is fitted to the discontinuities of $f$ ($120 \times 120$
structured) resulting in a very accurate solution ($\|u -
u_h\|_{L^2(\Omega)} \leq C (120)^{-5} \|f\|_{L^2(\Omega)}$) and on the other
hand on a mesh that is not fitted to the discontinuities of $f$
($110\times 110$ elements). The unfitted data results in spurious high
frequency oscillations with small amplitude in the high order finite element
solution, as can be seem in Figure \ref{P4_plot} (left fitted data and
right unfitted data). The $L^2$-norm of the difference of the fitted
and the unfitted solution is a good measure of the size of the
perturbation. It is $1.7\times 10^{-4}$.

First we fixed $\gamma=10^{-6}$
after a few steps of a line search algorithm, using the same stabilization
parameter for $s_1$ and $s_5$. We solved the problem using $6$
unstructured (Delaunay) meshes
with $20$, $30$, $40$, $60$, $80$ and $100$ elements on the domain side. The $L^2$-error
in $q$ is given in Figure \ref{graph}. Circle markers indicate the
result obtained with the stabilized method and square markers the
result obtained taking $\gamma=0$ above. In the left plot the data
$u_0$ is given by the accurate computation and in the right plot the
perturbed data is used. As can be seen in the plots, for the
unperturbed data the stabilized method performs slightly better than
the unstabilized method and has approximately $h^{\frac12}$ order
convergence in the $L^2$-norm, which is optimal. For the perturbed
data on the other hand the situation is dramatically different,
whereas the stabilized method almost has the same convergence, on
coarse meshes and only stagnates on finer meshes when the effect of
the perturbation becomes important, the unstabilized method diverges.

The exact and reconstructed source
function for the case of fitted data (on the $80\times 80$ mesh) is
given in Figure \ref{plot} and with unfitted data in Figure
\ref{plot2}.

\bibliographystyle{abbrv}
\bibliography{Cauchy}


\begin{figure}
\centering
\includegraphics[width=7cm]{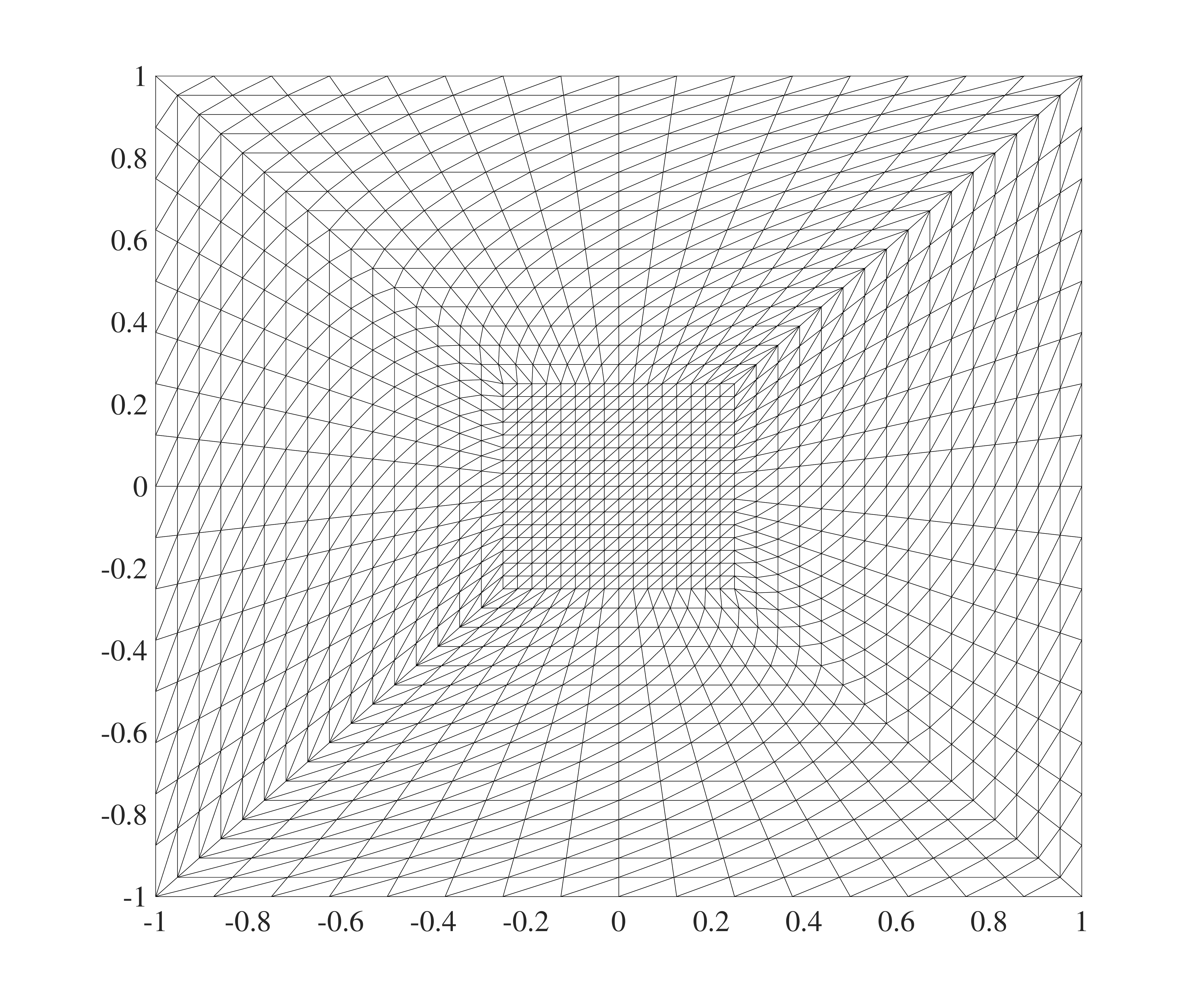}\includegraphics[width=7cm]{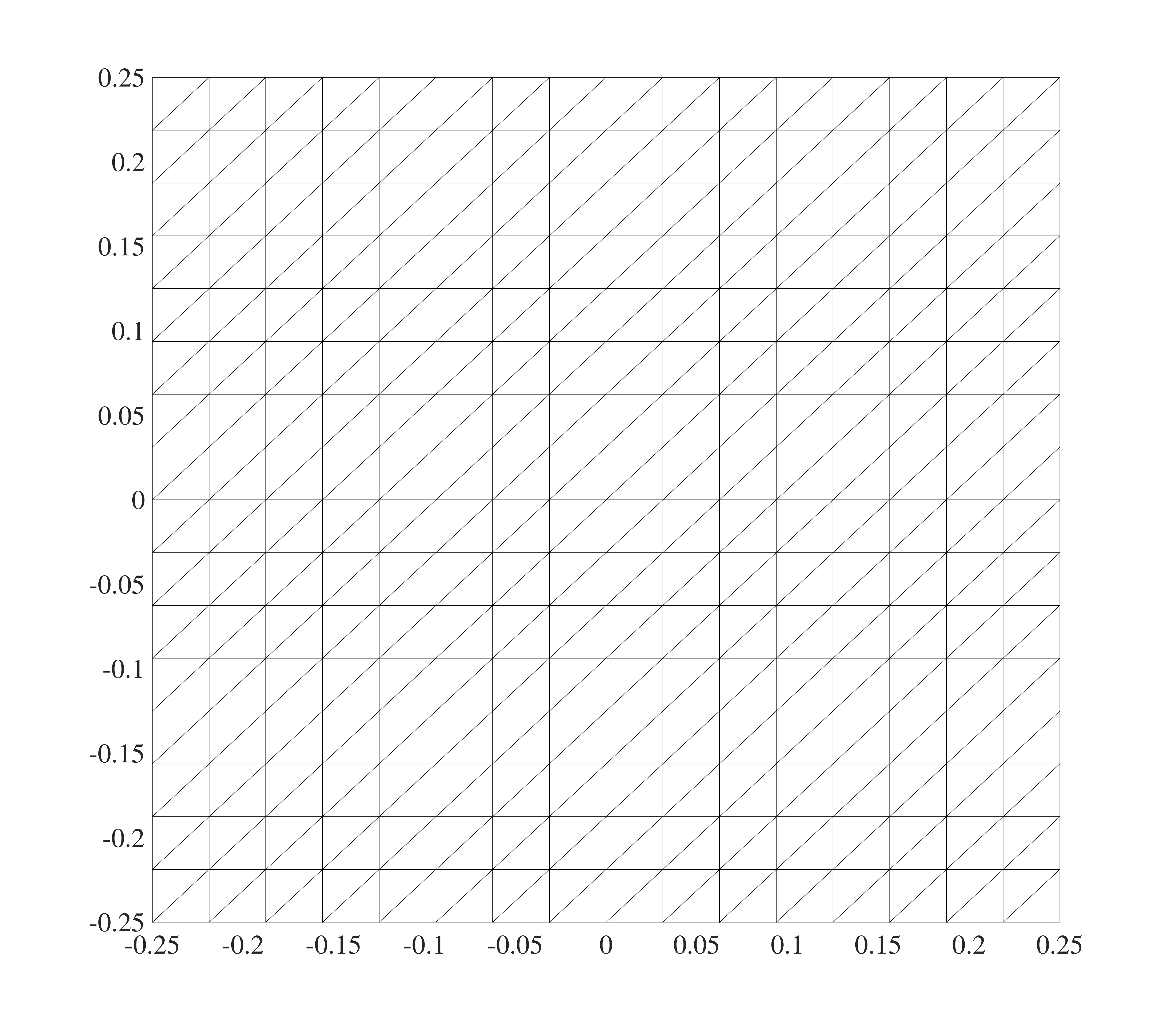}
\caption{Mesh on $\Omega$ and on $\mathcal{M}$.}\label{fig:meshes}
\end{figure}
\begin{figure}
\centering
\includegraphics[width=5cm]{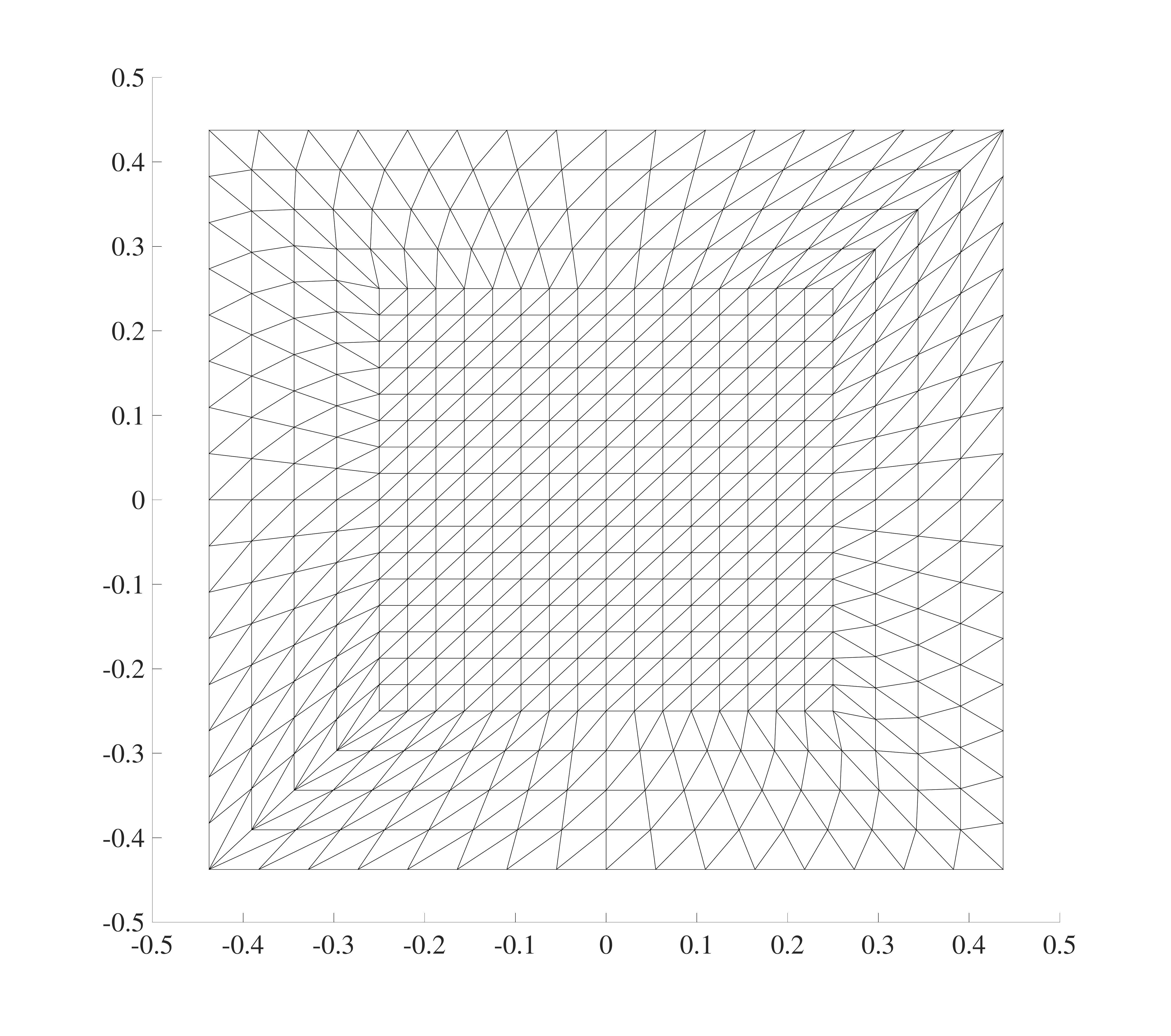}\includegraphics[width=5cm]{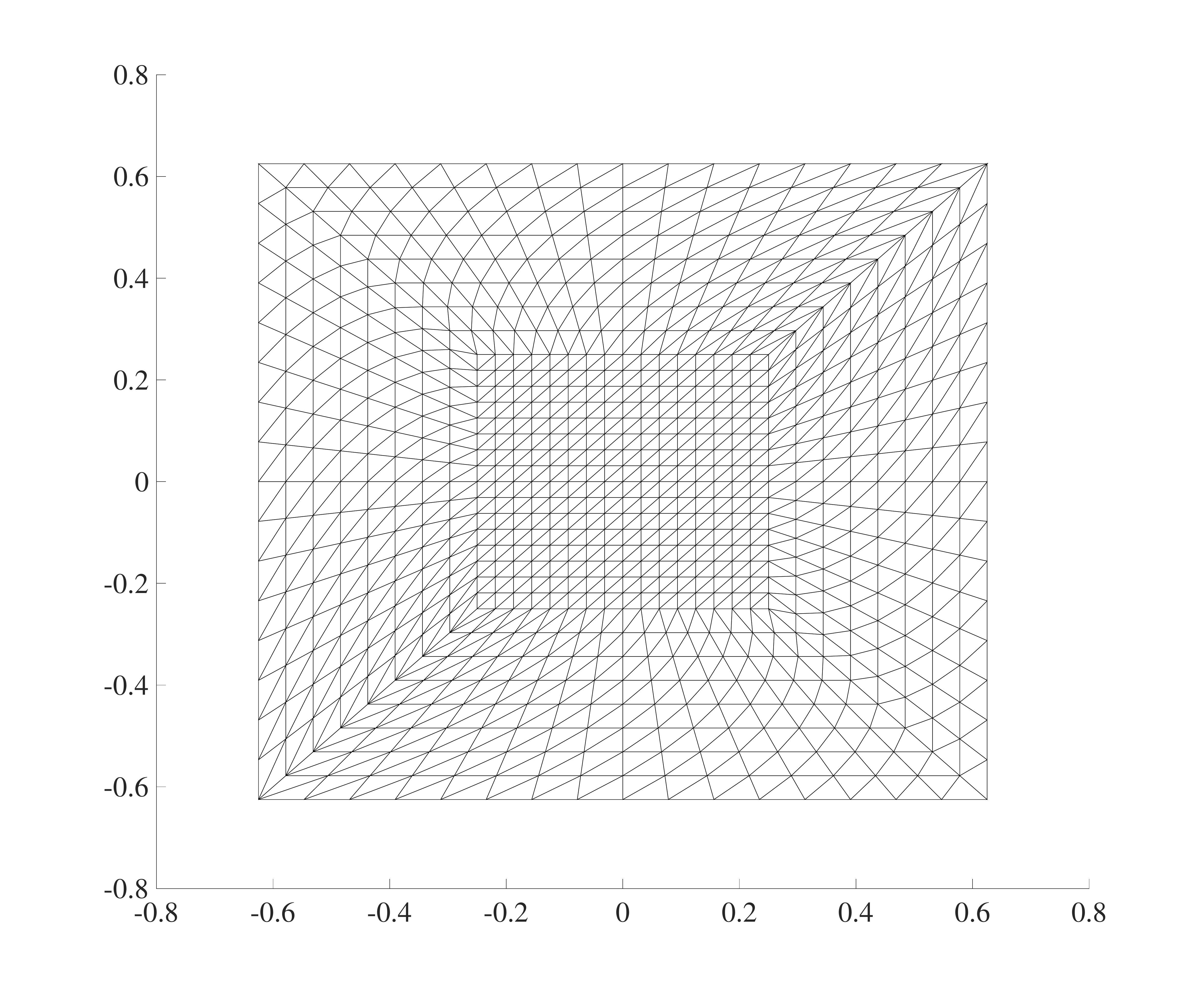}\includegraphics[width=5cm]{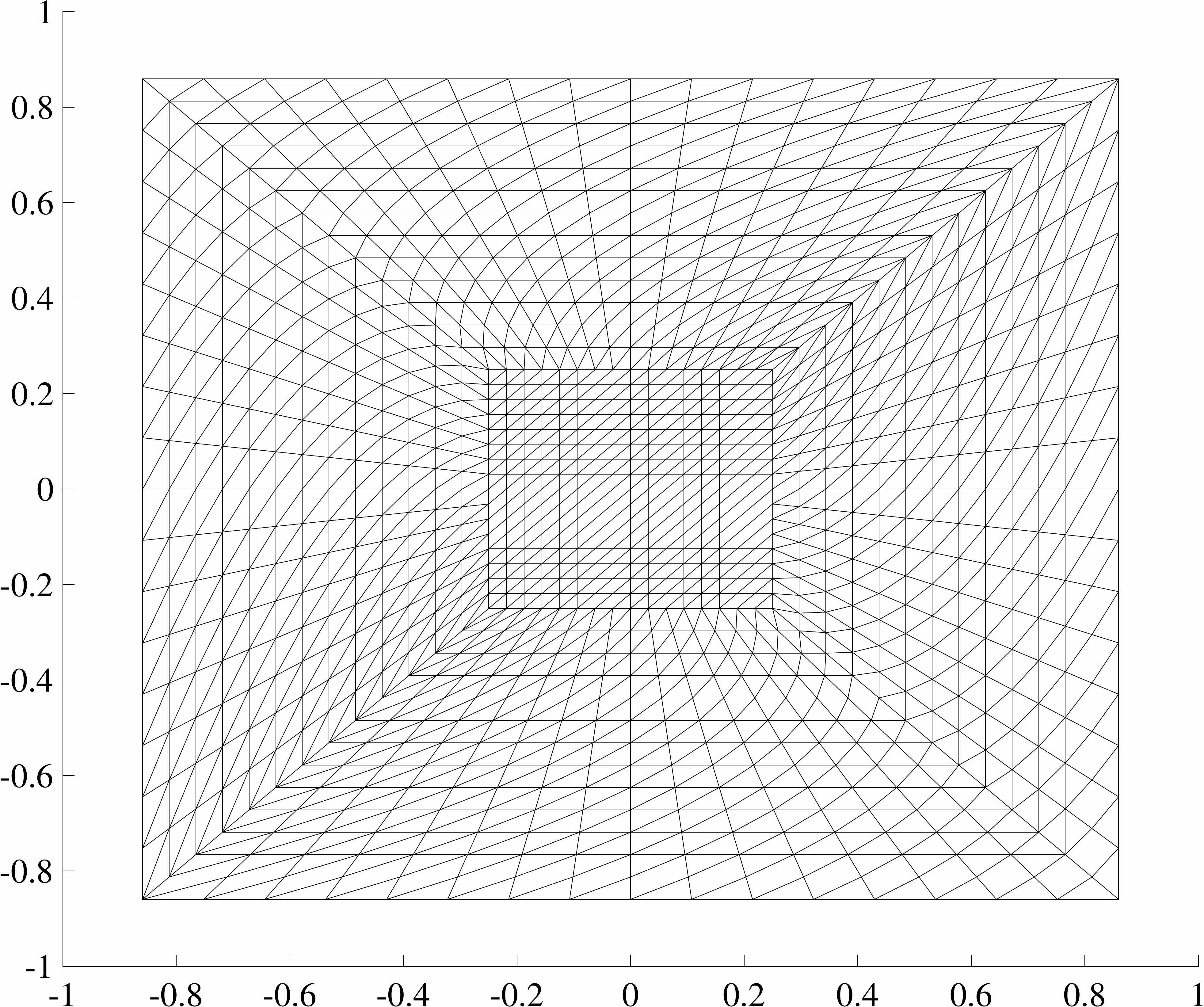}
\caption{Successively larger $B_{r_2}$; horizontal and verical distance from $\partial B_{r_2}$ to $\mathcal{M}$ is $0.1875$, $0.375$ and $0.625$.}\label{fig:Bdiff}
\end{figure}
\begin{figure}
\centering
\includegraphics[width=14cm]{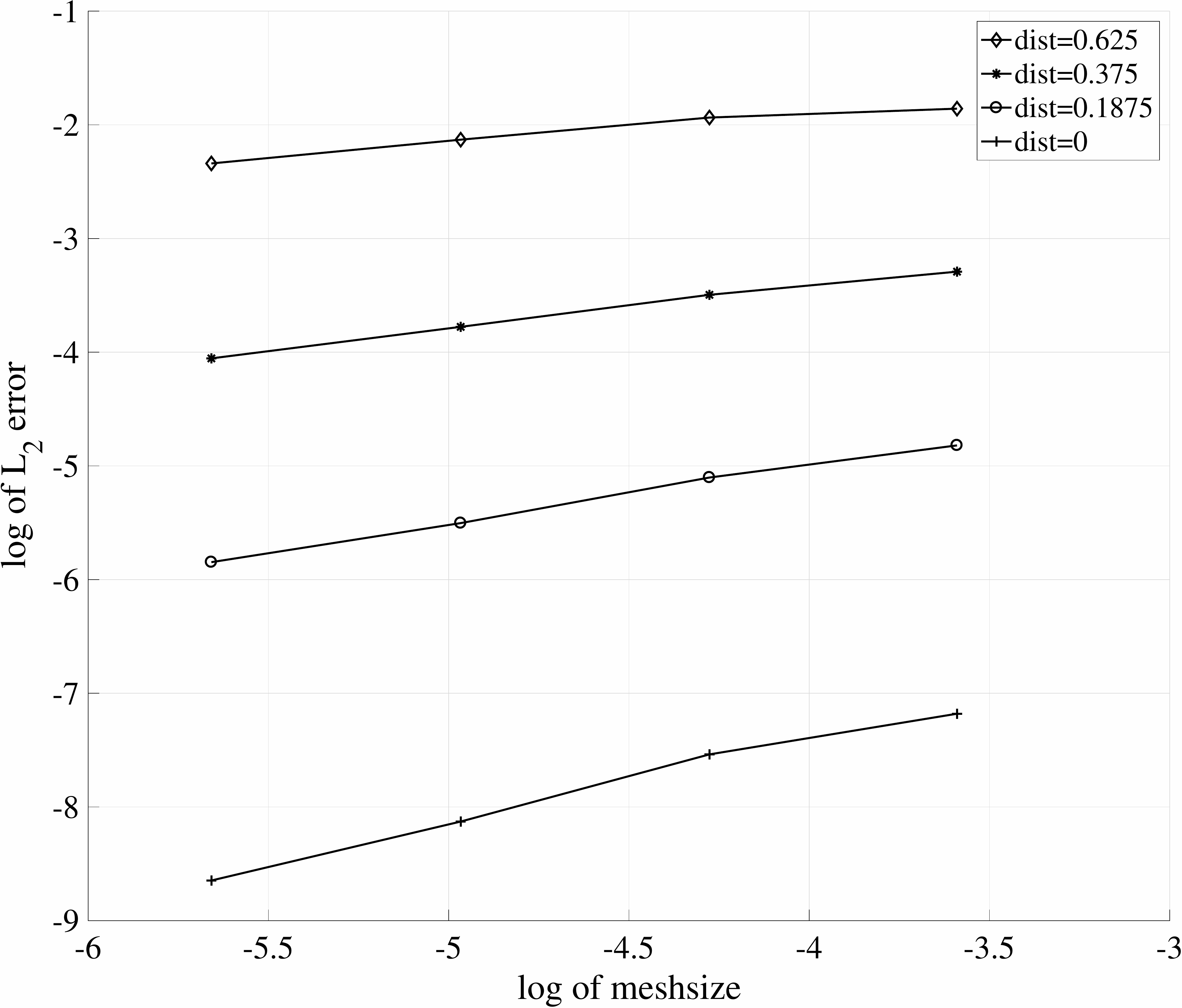}
\caption{Convergence for on different $B_{r_2}$.}\label{convassim}
\end{figure}
\begin{figure}
\centering
\includegraphics[width=7cm]{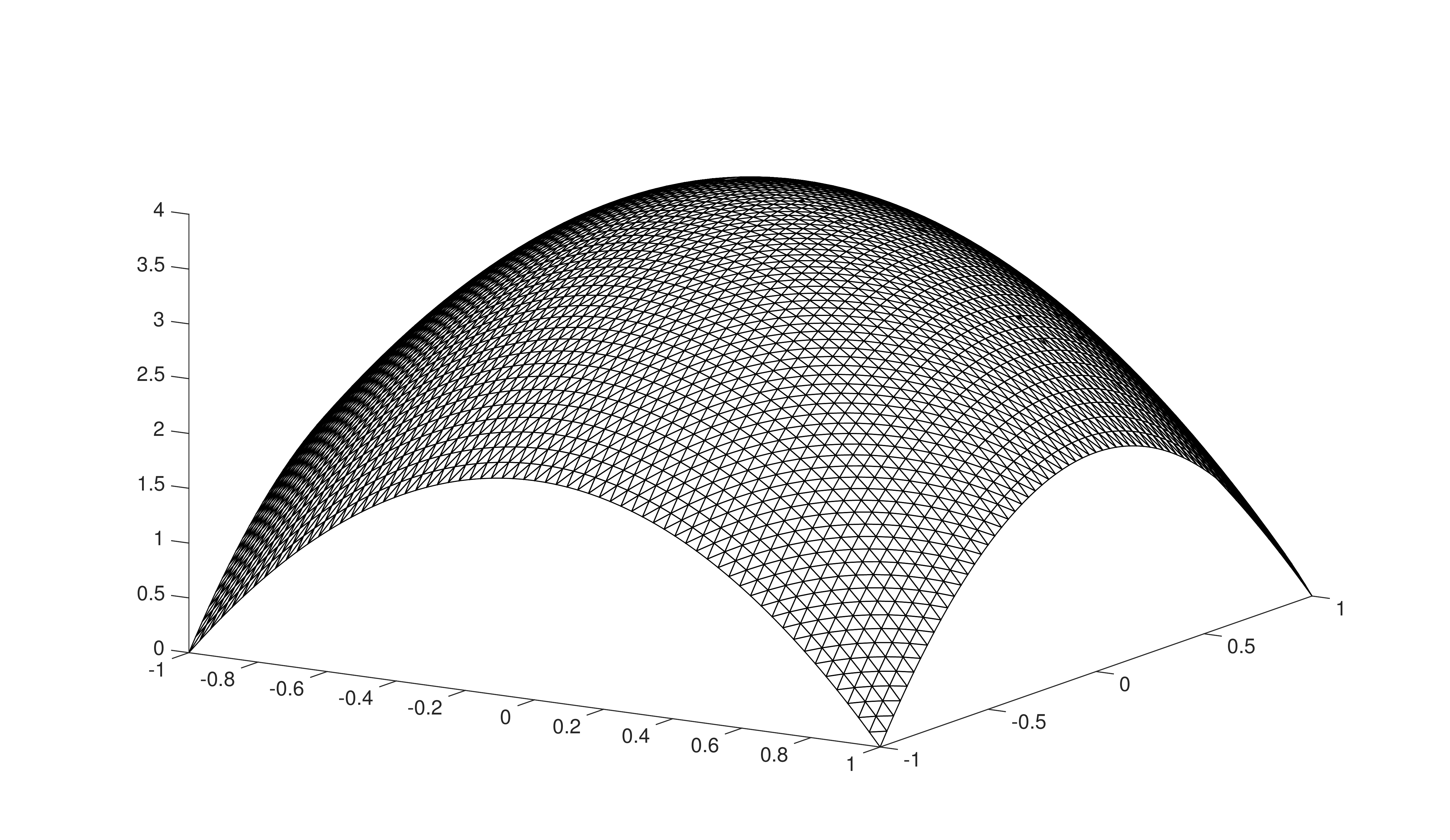}\includegraphics[width=7cm]{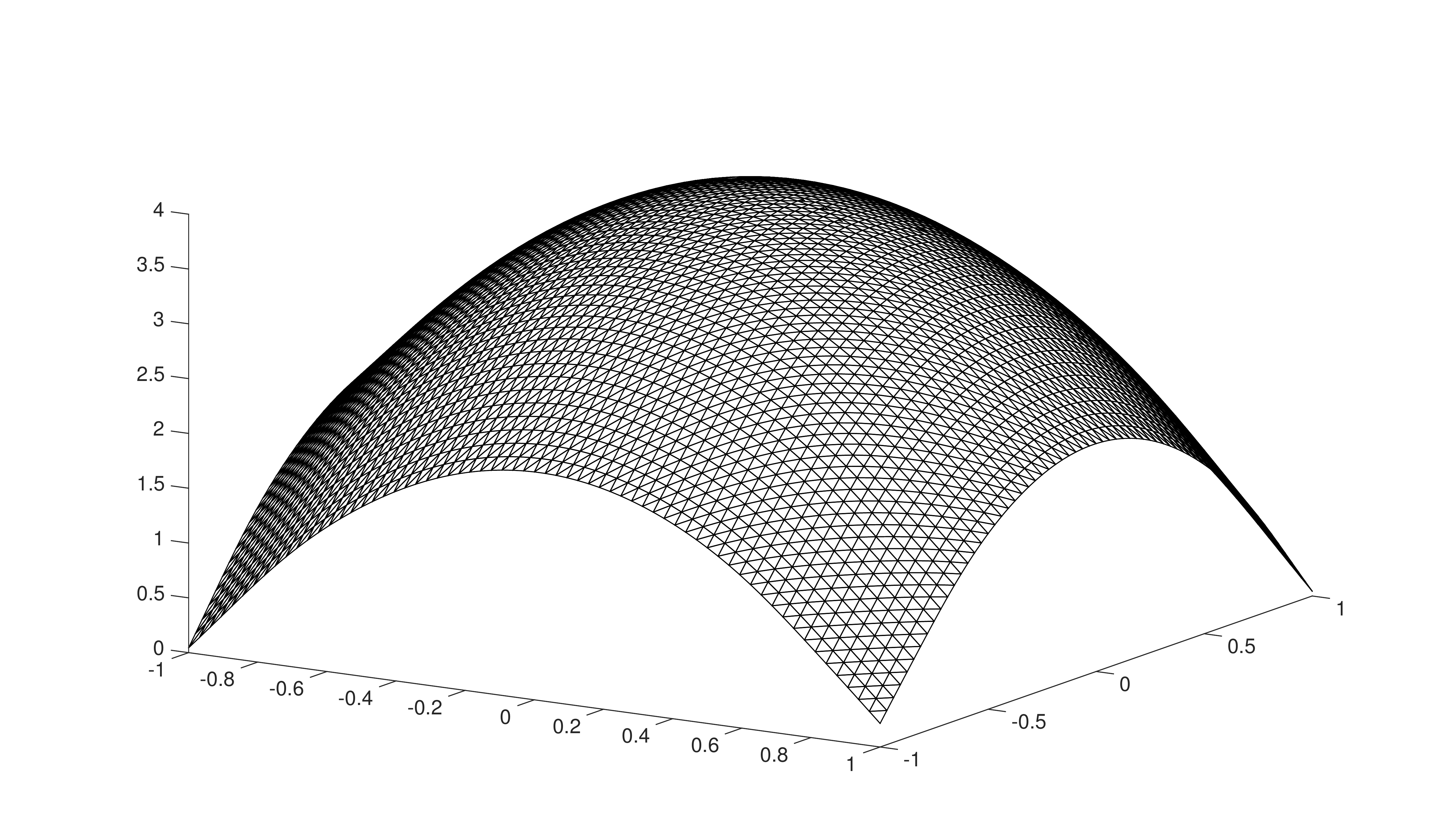}
\caption{Left: interpolated source term. Right: source term obtained with the present method.}\label{tikjump}
\end{figure}
\begin{figure}
\centering
\includegraphics[width=7cm]{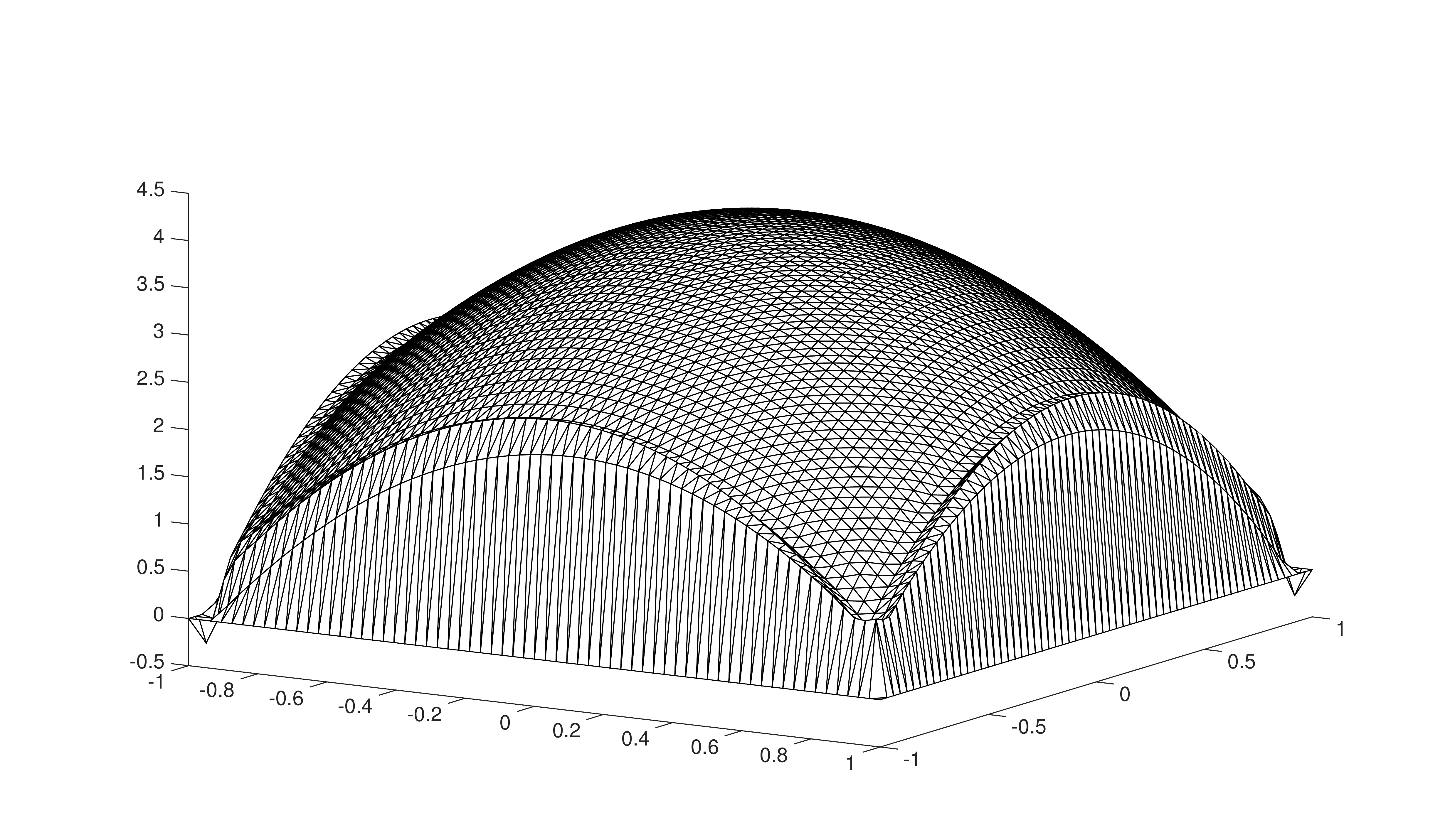}\includegraphics[width=7cm]{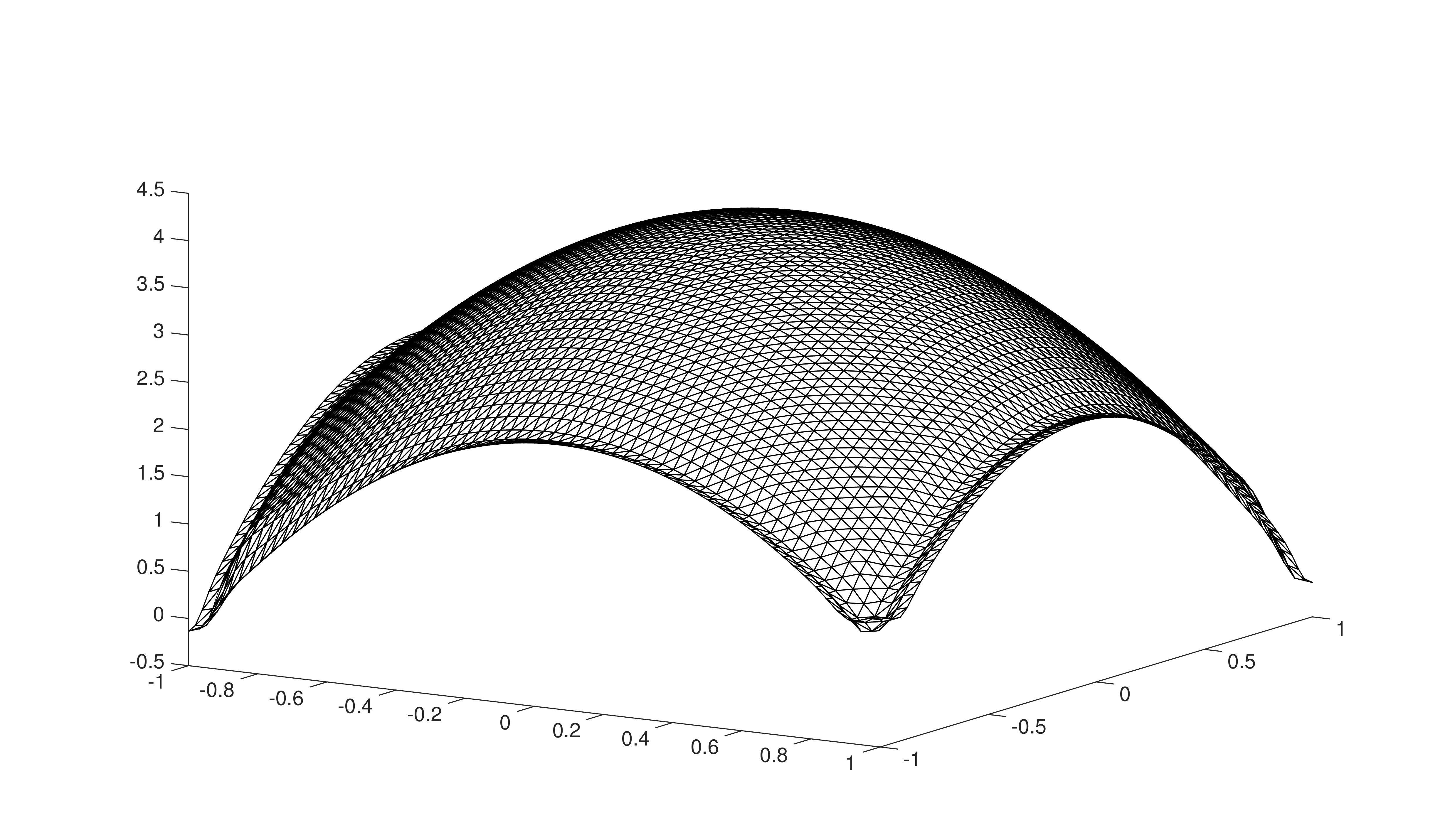}
\caption{Left: Tikhonov with zero order term. Right: Tikhonov with Laplacian.}\label{tikh1}
\end{figure}
\begin{figure}
\centering
\includegraphics[width=14cm]{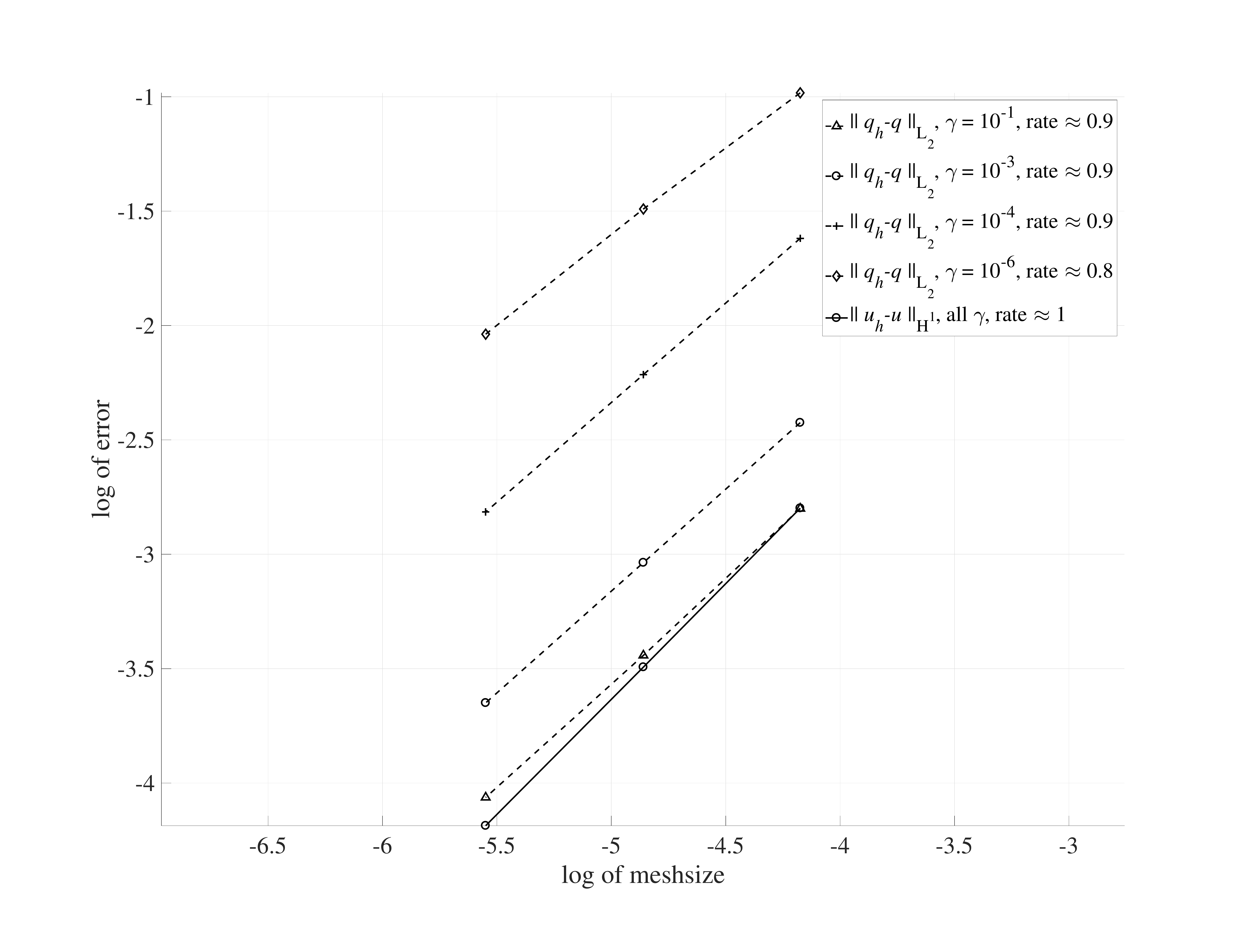}
\caption{Convergence for a smooth source term.}\label{smooth}
\end{figure}

\begin{figure}
\centering
\includegraphics[width=14cm]{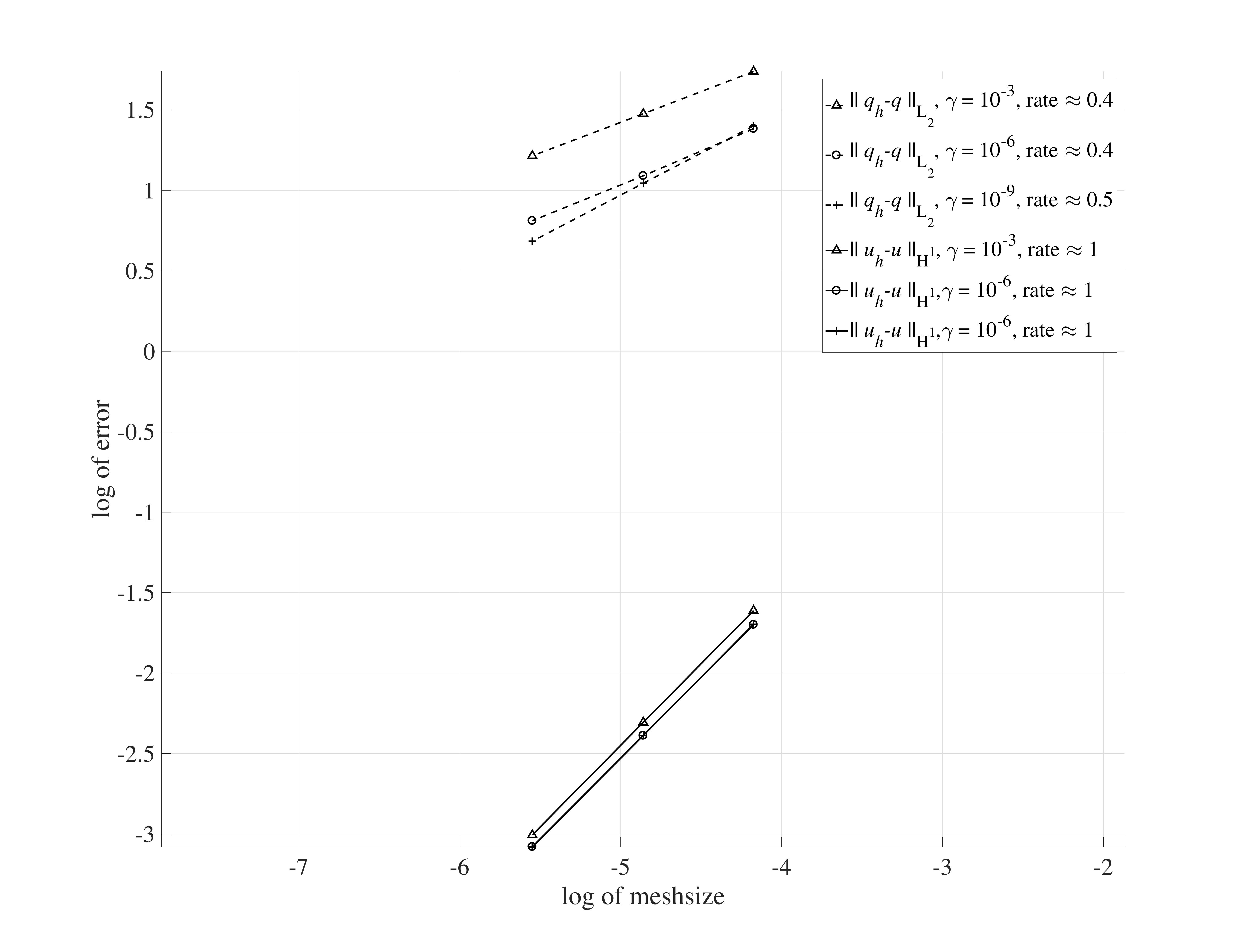}
\caption{Convergence for a non--smooth source term.}\label{nonsmooth}
\end{figure}

\begin{figure}
\centering
\includegraphics[width=14cm]{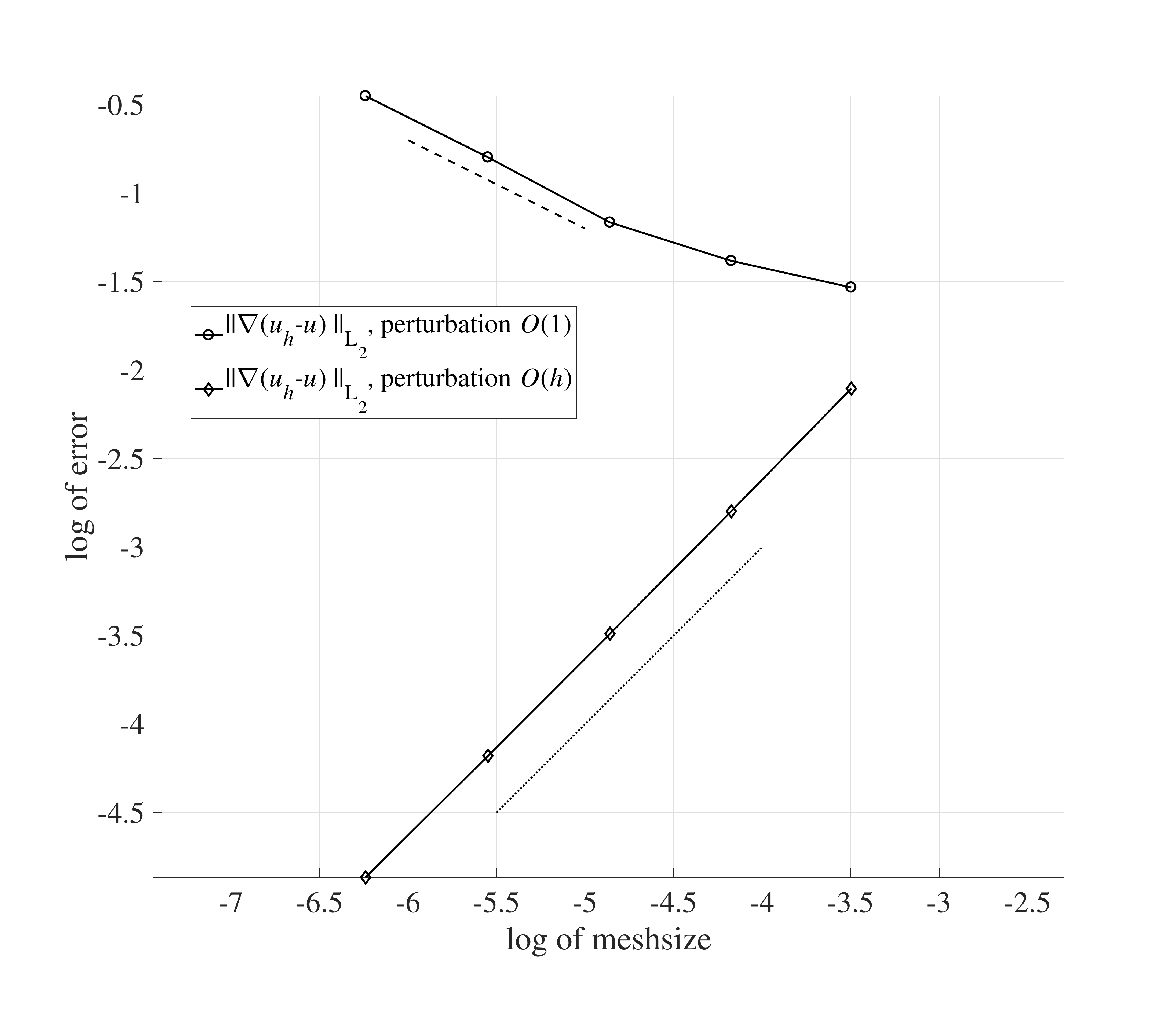}
\caption{Convergence for smooth source with perturbed data $u_0$. Dashed line has inclination -1/2, dotted line has inclination 1/1.}\label{perturbed}
\end{figure}

\begin{figure}
\centering{
\includegraphics[width=6cm]{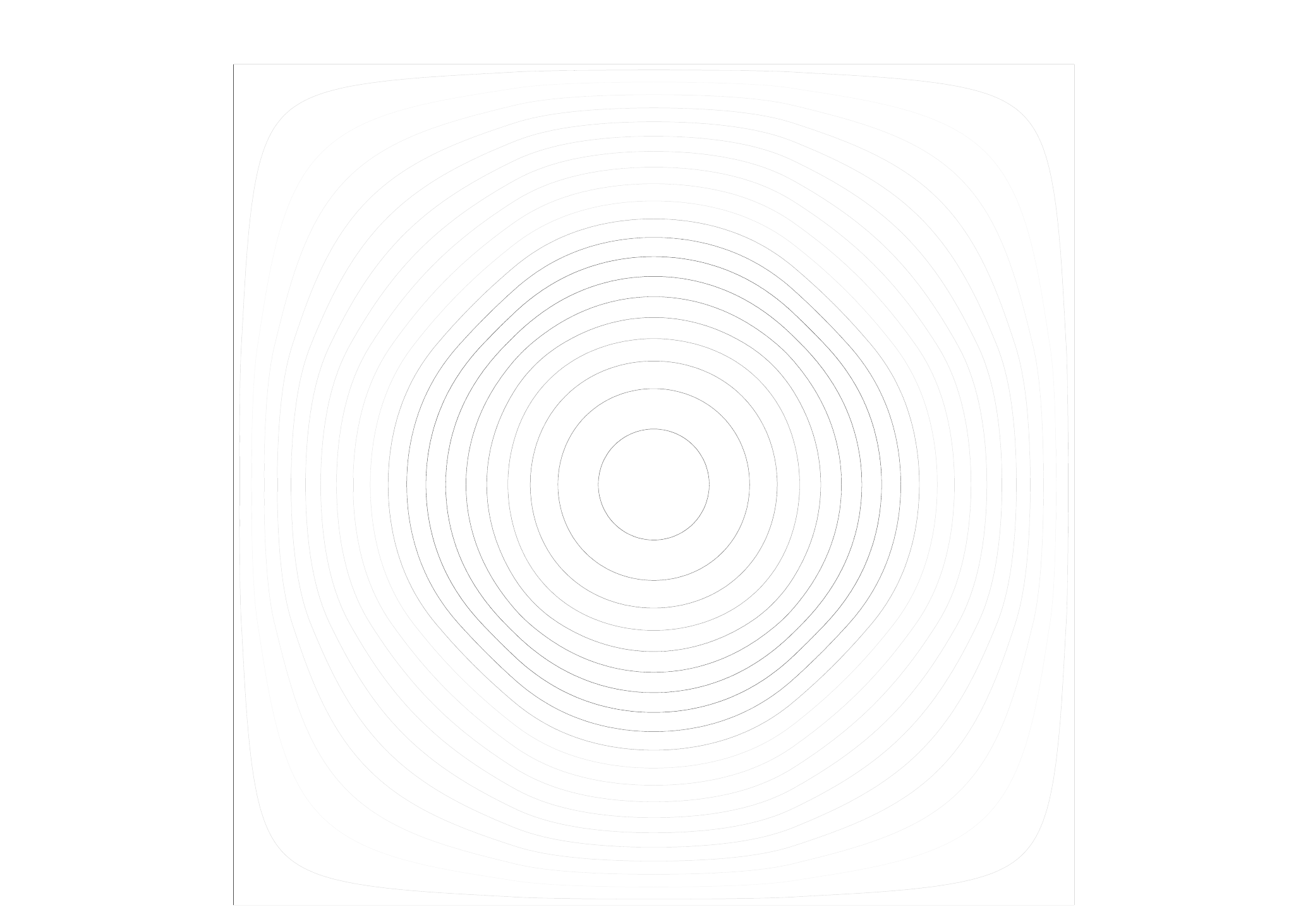}
\includegraphics[width=6cm]{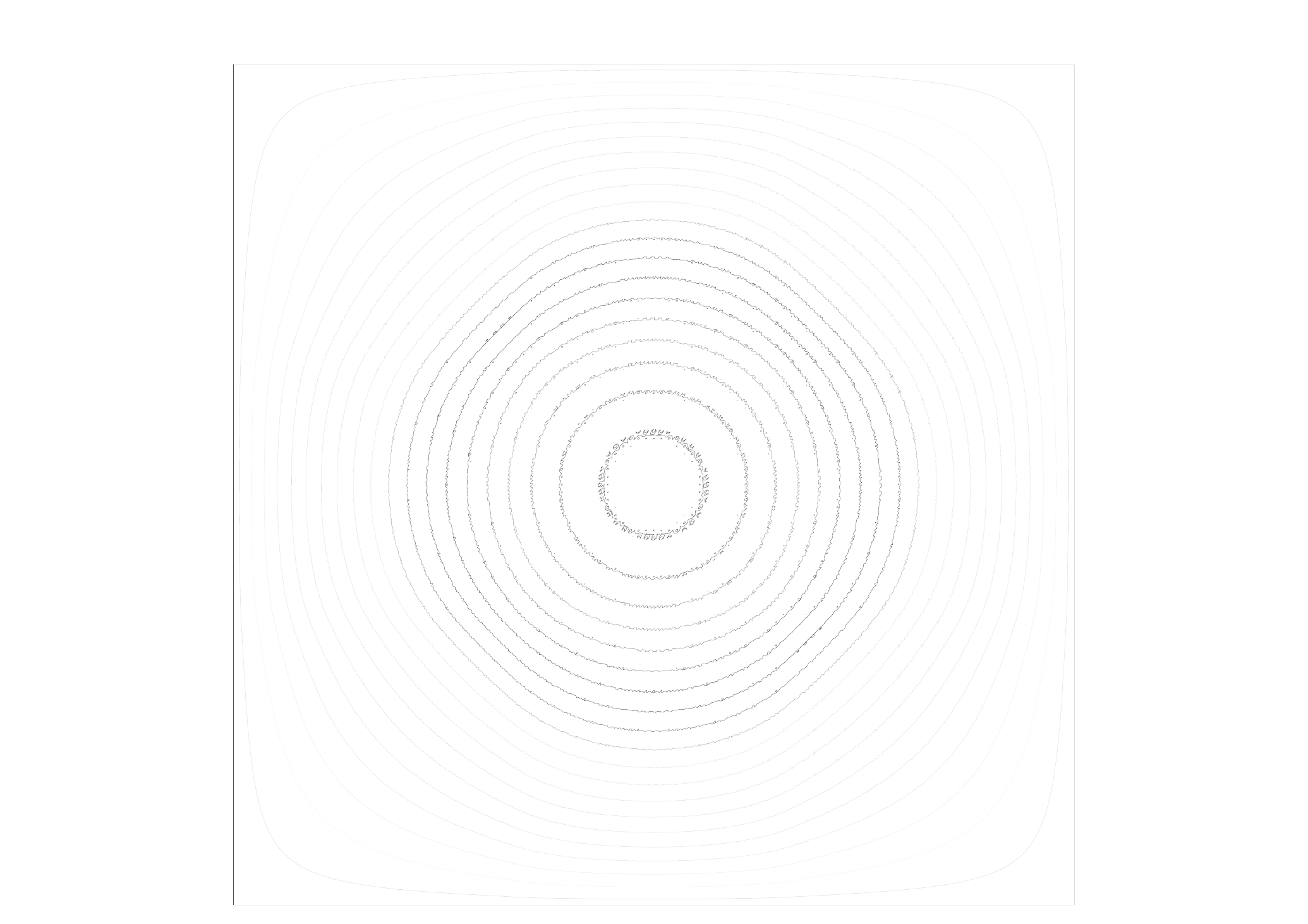}}
\caption{Left: data $u_0$ using fitted source term. Right: data $u_0$ using unfitted source term.}\label{P4_plot}
\end{figure}
\begin{center}
\begin{figure}
\centering
\hspace{-1cm}
\includegraphics[width=7cm]{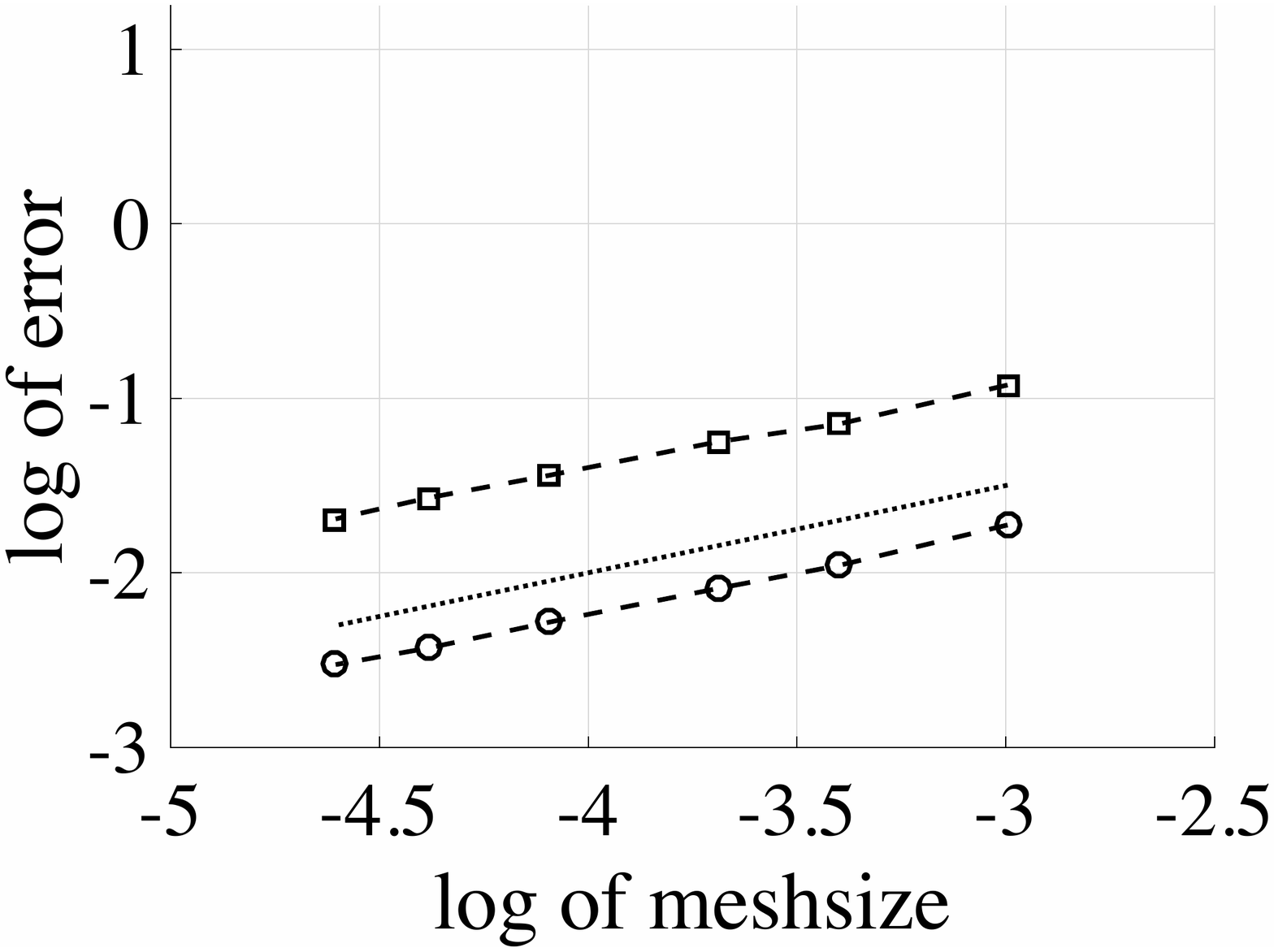}\hspace{-1cm}
\includegraphics[width=7cm]{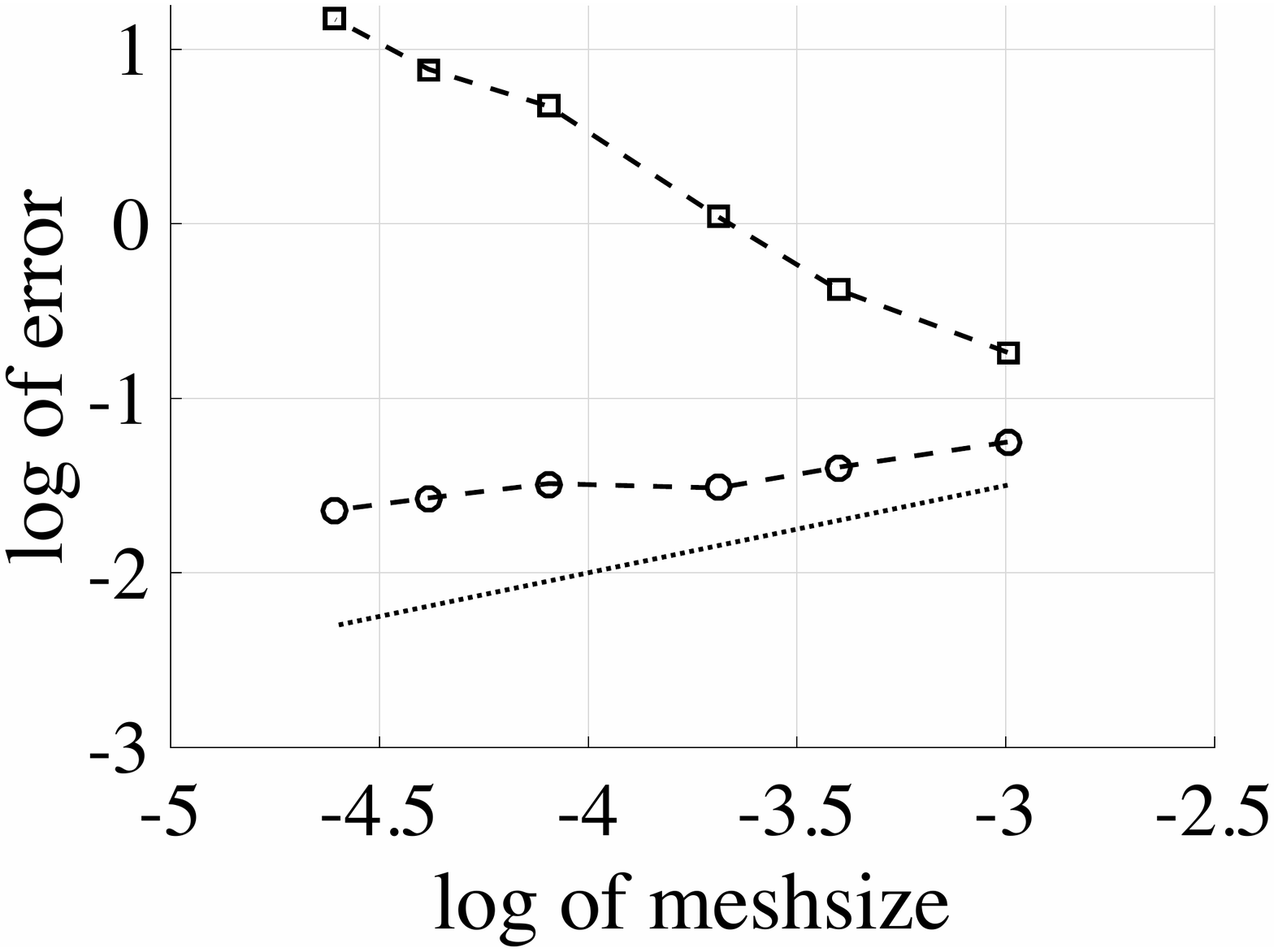}
\caption{Convergence plot of the $L^2$-norm error of the
  reconstruction. Circle markers denote stabilized formulation and square
  markers unstabilized. Left: unperturbed data. Right: perturbed
  data. Dotted line is $h^{\frac12}$ and same in both
  graphics. Filled line in right plot is $0.05 h^{-1}$}\label{graph}
\end{figure}
\begin{figure}
\centering
\hspace{-1cm}\includegraphics[width=5cm]{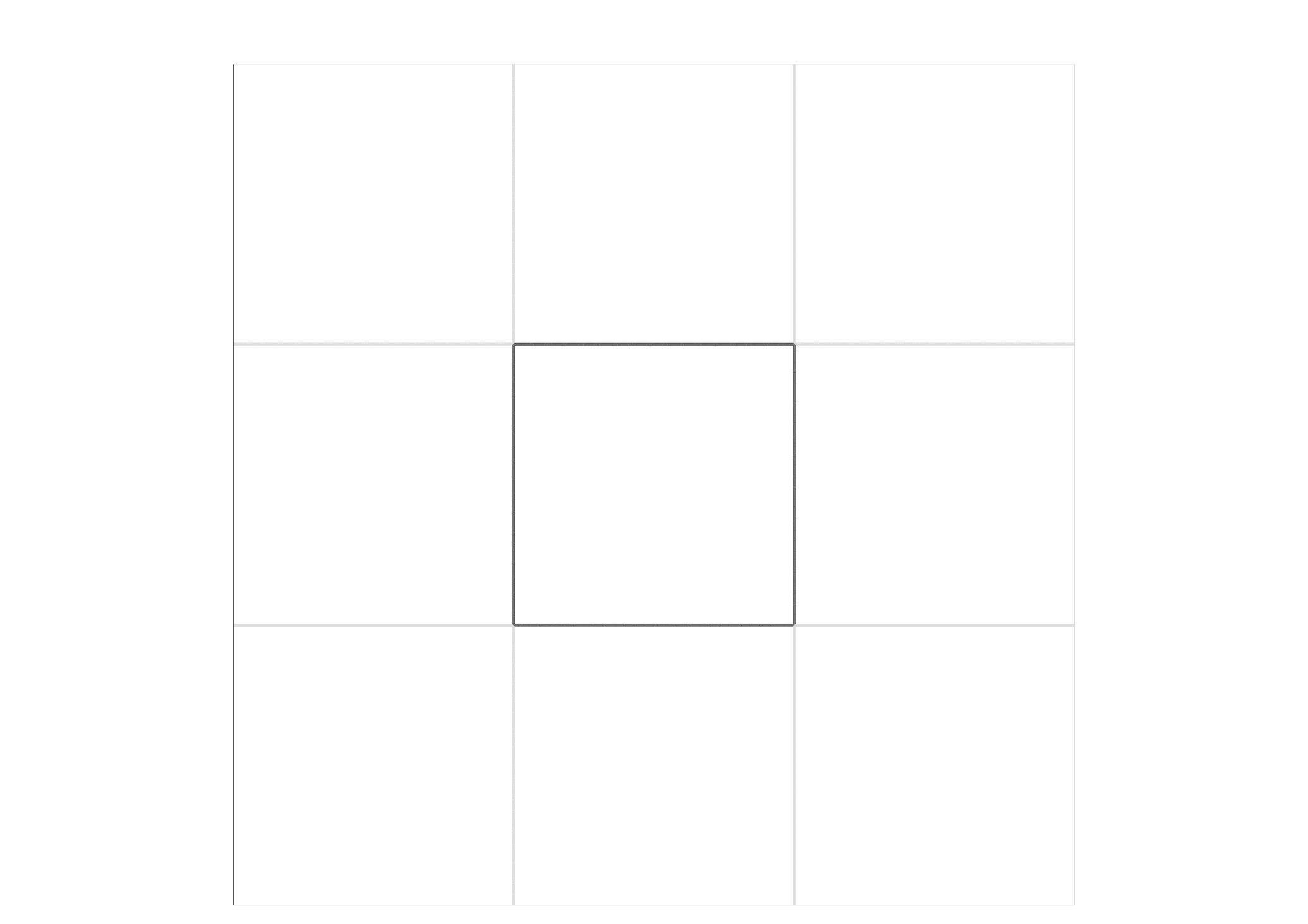}\hspace{-1cm}
\includegraphics[width=5cm]{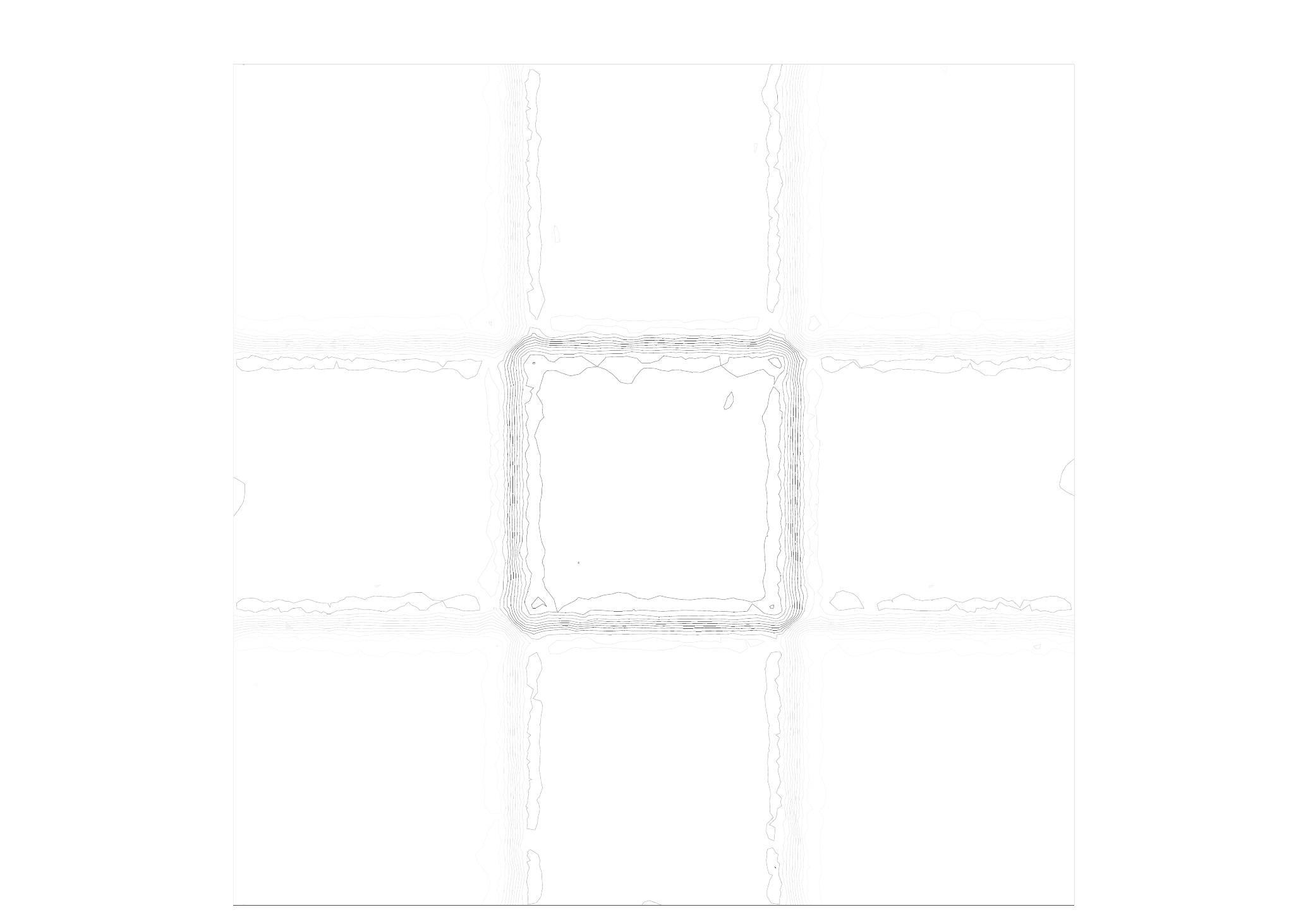}\hspace{-1cm}
\includegraphics[width=5cm]{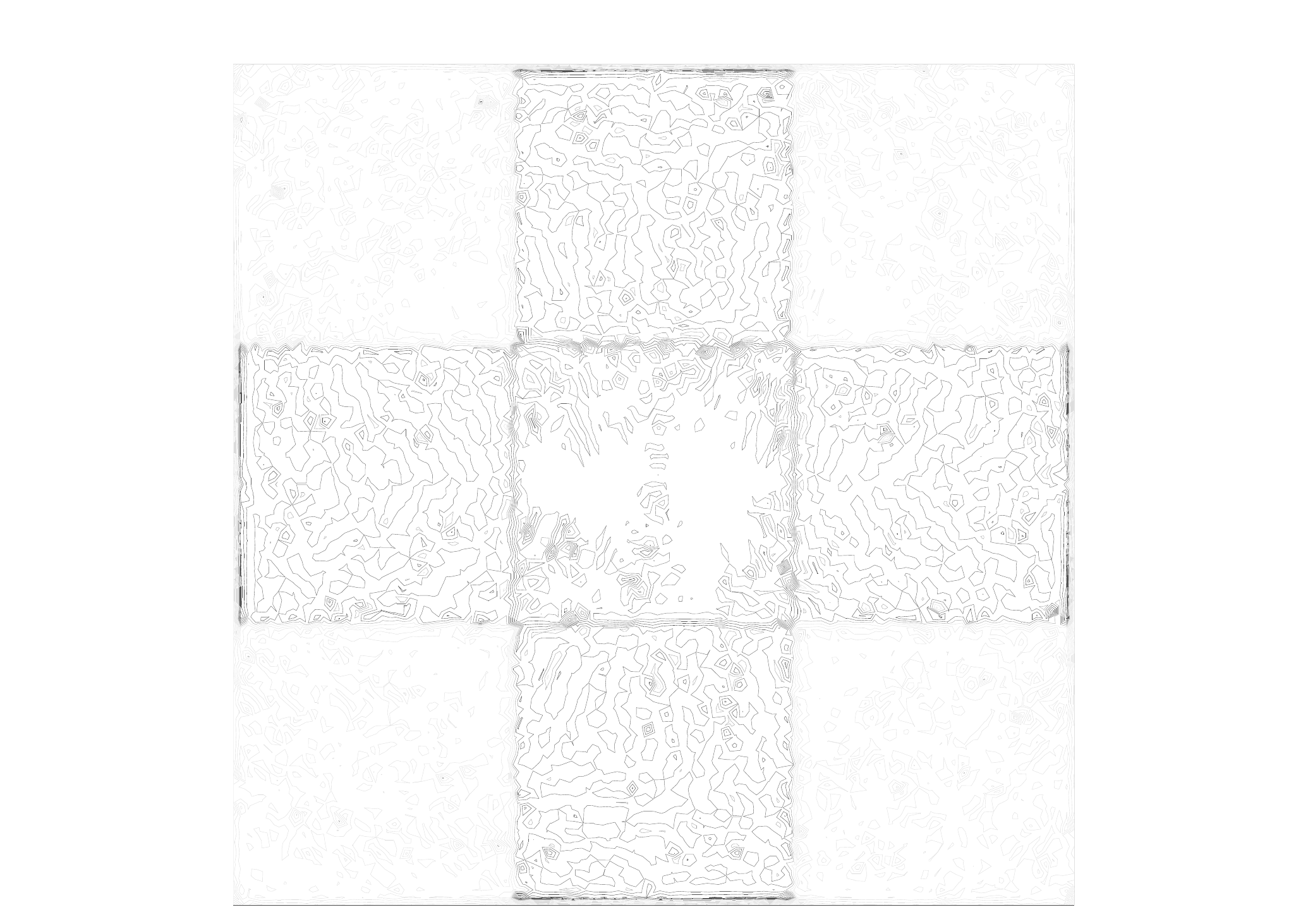}
\caption{Contout lines of exact and reconstructed source terms. Left: exact source term. Middle: reconstructed source
term using stabilization, unperturbed data. Right: unstabilized reconstruction, unperturbed data.}\label{plot}
\end{figure}

\begin{figure}
\centering
\includegraphics[width=5cm]{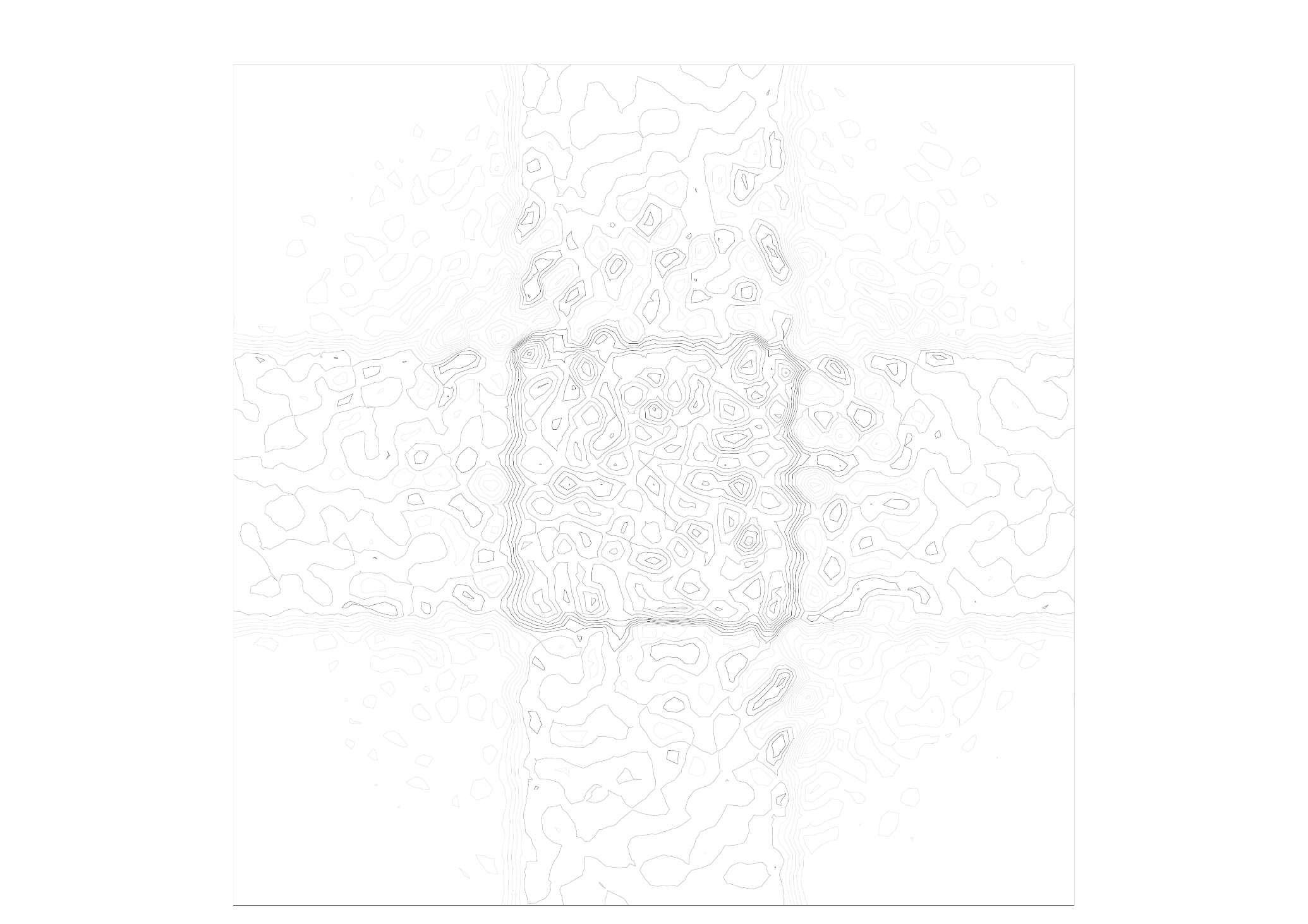}\hspace{-1cm}
\includegraphics[width=5cm]{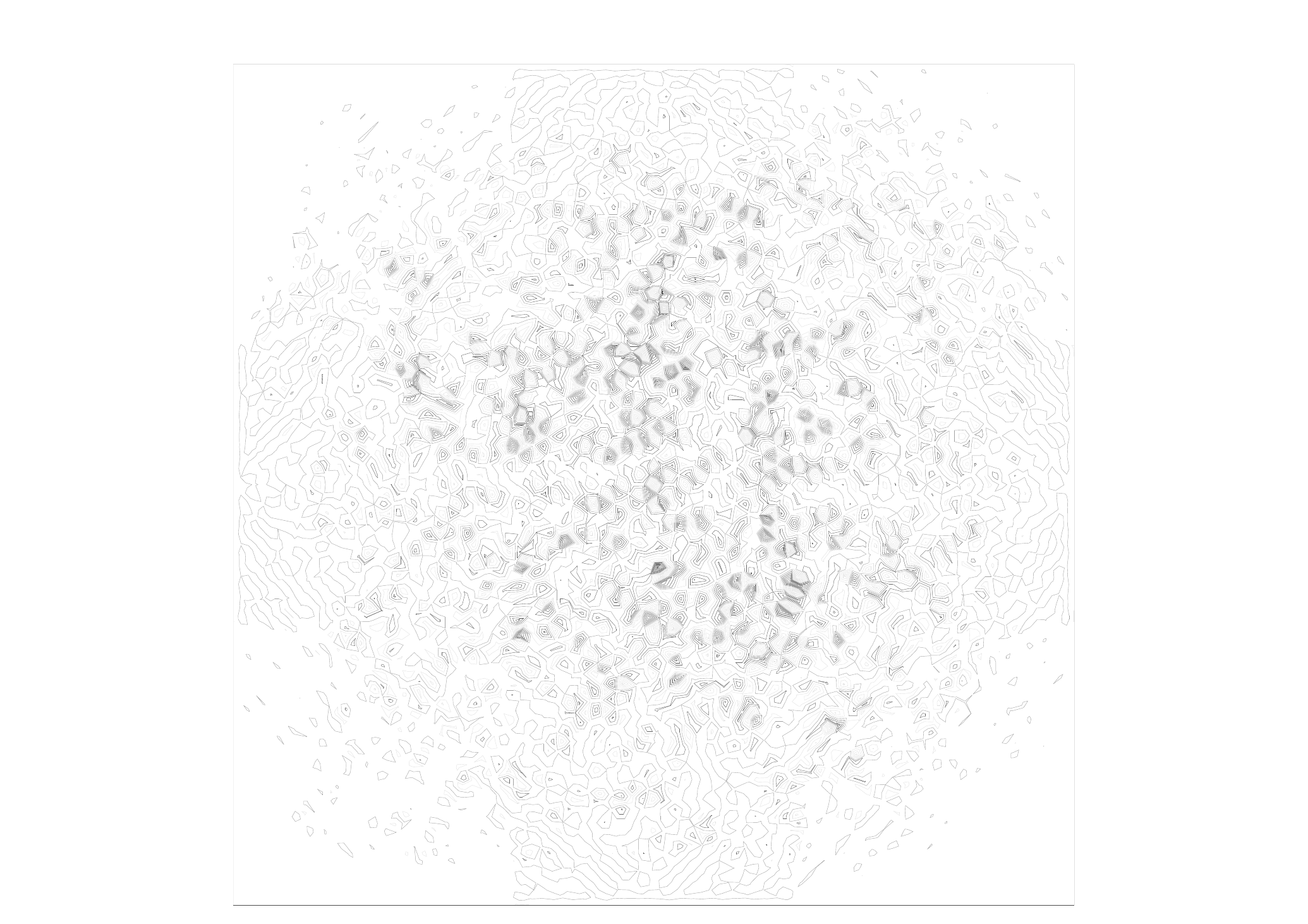}
\caption{Left: reconstructed source
term using stabilization, perturbed data. Right: unstabilized reconstruction, perturbed data.}\label{plot2}
\end{figure}

\end{center}

\end{document}